\documentclass[12pt]{amsart}

\usepackage{amsmath,amsfonts,amsthm,amssymb,amscd}
\usepackage{amssymb}
\usepackage[utf8]{inputenc}
\usepackage[T2A]{fontenc}
\usepackage[all]{xy}
\usepackage{enumitem}
\usepackage[english]{babel}
\usepackage{graphicx}
\usepackage{xcolor}
\usepackage[margin=2.5cm]{geometry}

\usepackage{enumitem}

\usepackage{hyperref}
\hypersetup{
    hypertexnames=false,
    colorlinks,
    linkcolor={red!50!black},
    citecolor={blue!50!black},
    urlcolor={blue!80!black}
}

\newtheorem{mytheorem}[subsection]{Theorem}
\newtheorem{mylemma}[subsection]{Lemma}

\newtheorem{mydefinition}[subsection]{Definition}

\newtheorem{myremark}[subsection]{Remark}
\newtheorem{myexample}[subsection]{Example}

\makeatletter
\@addtoreset{subsection}{section}
\@addtoreset{equation}{section}
\@addtoreset{figure}{section}
\@addtoreset{table}{section}
\makeatother

\makeatletter
\newcommand\testshape{family=\f@family; series=\f@series; shape=\f@shape.}
\def\myemphInternal#1{\if n\f@shape%
\begingroup\itshape #1\endgroup\/%
\else\begingroup\sf\itshape #1\endgroup%
\fi}
\def\myemph{\futurelet\testchar\MaybeOptArgmyemph}
\def\MaybeOptArgmyemph{\ifx[\testchar \let\next\OptArgmyemph
                 \else \let\next\NoOptArgmyemph \fi \next}
\def\OptArgmyemph[#1]#2{\index{#1}\myemphInternal{#2}}
\def\NoOptArgmyemph#1{\myemphInternal{#1}}
\makeatother
\newcommand\id{\mathrm{id}}          % identity map    
\newcommand\Int{\mathrm{Int}}        % interior 
\newcommand\Fix[1]{\mathrm{Fix}(#1)} % fixed points
\newcommand\rank{\mathrm{rank}}      % rank
\newcommand\bR{\mathbb{R}}
\newcommand\bZ{\mathbb{Z}}
\newcommand\bN{\mathbb{N}}

\newcommand\Aut{\mathrm{Aut}}       % automorphisms
\newcommand\Diff{\mathcal{D}}       % diffeomorphisms
\newcommand\Homeo{\mathcal{H}}      % homeomorphisms
\newcommand\Map{\mathrm{Map}}       % maps
\newcommand\Orb{\mathcal{O}}        % orbit
\newcommand\Stab{\mathcal{S}}       % stabilizer
\newcommand\DiffId{\Diff_{\id}}     % diffeomorphisms
\newcommand\StabId{\Stab_{\id}}     % stabilizer

\newcommand\Cinfty{\mathcal{C}^{\infty}}
\newcommand\Ci[2]{\mathcal{C}^{\infty}(#1,#2)}               % space of C^\infty maps
\newcommand\Morse[2]{\mathcal{M}(#1,#2)}                     % space of Morse maps
\newcommand\MorseSmp[2]{\mathcal{M}^{smp}(#1,#2)}            % space of simple Morse maps
\newcommand\MorseGen[2]{\mathcal{M}^{gen}(#1,#2)}            % space of generic Morse maps

\newcommand\Stabilizer[1]{\Stab(#1)}
\newcommand\StabilizerId[1]{\StabId(#1)}
\newcommand\StabilizerIsotId[1]{\Stab'(#1)}
\newcommand\Orbit[1]{\Orb(#1)}
\newcommand\OrbitComp[2]{\Orb_{#1}(#2)}

\newcommand\FolStabilizer[1]{\Delta(#1)}

\newcommand\FolStabilizerIsotId[1]{\Delta'(#1)}

\newcommand\SO{\mathrm{SO}}
\newcommand\UnitGroup{\{1\}}

\newcommand\Bman{B}

\newcommand\Kman{K}

\newcommand\Mman{M}

\newcommand\Pman{P}

\newcommand\Rman{R}

\newcommand\Xman{X}
\newcommand\Yman{Y}
\newcommand\Zman{Z}
\newcommand\DiffM{\Diff(\Mman)}
\newcommand\DiffIdM{\DiffId(\Mman)}

\newcommand\func{f}
\newcommand\gfunc{g}
\newcommand\dif{h}

\newcommand\Grp{\mathbf{G}}
\newcommand\Grpf[1]{\Grp(#1)}
\newcommand\Gf{\Grpf{\func}}

\newcommand\Ga{\Grpf{\alpha}}
\newcommand\Gb{\Grpf{\beta}}

\newcommand\AdmisWords{L^{*}}

\newcommand\PClassAll{\mathcal{P}}
\newcommand\PClassTwo{\PClassAll_2}

\newcommand\ClassGf{\mathbf{G}}
\newcommand\ClassGfM{\ClassGf(\Mman,\Pman)}
\newcommand\ClassGfSmpM{\ClassGf^{smp}(\Mman,\Pman)}
\newcommand\ClassGfGenM{\ClassGf^{gen}(\Mman,\Pman)}

\newcommand\KRGraph[1]{\Gamma_{#1}}
\newcommand\KRGraphf{\KRGraph{\func}}

\newcommand\word{W}
\newcommand\aword{A}
\newcommand\bword{B}
\newcommand\len{\mathrm{len}}

\newcommand\RR{\mathcal{R}}
\newcommand\GG{\mathcal{G}}

\newcommand\bSet{X}
\newcommand\bGroup{G}
\newcommand\actHom{\delta}
\newcommand\aGroup{H}
\newcommand\hh{h}
\newcommand\myqed{}

\begin{document}

\title[Automorphisms of Kronrod-Reeb graphs]{Automorphisms of Kronrod-Reeb graphs of Morse functions on compact surfaces}
%\title{Automorphisms of Kronrod-Reeb graphs of Morse functions on compact surfaces}

\author{Anna Kravchenko}
\email{annakravchenko1606@gmail.com}
\address{Taras Shevchenko National University of Kyiv, Ukraine}

\author{Sergiy Maksymenko}
\email{maks@imath.kiev.ua}
\address{Institute of Mathematics, National Academy of Sciences of Ukraine, Kyiv, Ukraine}



\begin{abstract}
Let $M$ be a connected orientable compact surface, $f:M\to\mathbb{R}$ be a Morse function, and $\mathcal{D}_{\mathrm{id}}(M)$ be the group of difeomorphisms of $M$ isotopic to the identity.
Denote by $\mathcal{S}'(f)=\{f\circ h = f\mid h\in\mathcal{D}_{\mathrm{id}}(M)\}$ the subgroup of $\mathcal{D}_{\mathrm{id}}(M)$ consisting of difeomorphisms ``preserving'' $f$, i.e. the stabilizer of $f$ with respect to the right action of $\mathcal{D}_{\mathrm{id}}(M)$ on the space $\mathcal{C}^{\infty}(M,\mathbb{R})$ of smooth functions on $M$.
Let also $\mathbf{G}(f)$ be the group of automorphisms of the Kronrod-Reeb graph of $f$ induced by diffeomorphisms belonging to $\mathcal{S}'(f)$.
This group is an important ingredient in determining the homotopy type of the orbit of $f$ with respect to the above action of $\mathcal{D}_{\mathrm{id}}(M)$ and it is trivial if $f$ is ``generic'', i.e. has at most one critical point at each level set $f^{-1}(c)$, $c\in\mathbb{R}$.

For the case when $M$ is distinct from $2$-sphere and $2$-torus we present a precise description of the family $\mathbf{G}(M,\mathbb{R})$ of isomorphism classes of groups $\mathbf{G}(f)$, where $f$ runs over all Morse functions on $M$, and of its subfamily $\mathbf{G}^{smp}(M,\mathbb{R}) \subset \mathbf{G}(M,\mathbb{R})$ consisting of groups corresponding to simple Morse functions, i.e. functions having at most one critical point at each connected component of each level set.

In fact, $\mathbf{G}(M,\mathbb{R})$, (resp. $\mathbf{G}^{smp}(M,\mathbb{R})$), coincides with the minimal family of isomorphism classes of groups containing the trivial group and closed with respect to direct products and also with respect to wreath products ``from the top'' with arbitrary finite cyclic groups, (resp. with group $\mathbb{Z}_2$ only).
\end{abstract}

\maketitle

\section{Introduction}\label{sect:introduction}

Let $\Mman$ be a smooth ($\Cinfty$) compact two-dimensional manifold, $\Pman$ be either the real line $\bR$ or the circle $S^1$, and $\DiffM$ be the group of $\Cinfty$ diffeomorphisms of $\Mman$.
Then there is a natural right action $\eta:\Ci{\Mman}{\Pman}\times\DiffM\to \Ci{\Mman}{\Pman}$ of this group on the space of smooth maps $\Mman\to\Pman$ defined by the formula $\eta(\func,\dif) = \func \circ \dif$.
For $\func \in \Ci{\Mman}{\Pman}$ let
\begin{align*}
\Stabilizer{\func} &=\{\dif\in\DiffM \mid \func \circ \dif = \func\}, &
\Orbit{\func} &=\{\func \circ \dif \mid \dif \in \DiffM\},
\end{align*}
be respectively its \myemph{stabilizer} and \myemph{orbit}.

We will endow all spaces of smooth maps, in particular $\DiffM$, $\Ci{\Mman}{\Pman}$, and all their subspaces, with the corresponding $\Cinfty$ topologies.
Denote by $\StabilizerId{\func}$, (respectively $\DiffIdM$), the path component of the identity map $\id_{\Mman}$ in $\Stabilizer{\func}$ (respectively, in $\DiffM$), and by $\OrbitComp{\func}{\func}$ the path component of $\func$ in $\Orbit{\func}$.
Let also
\[
\StabilizerIsotId{\func} = \Stabilizer{\func} \cap \DiffIdM,
\]
be the \myemph{stabilizer of $\func$ with respect to the induced action of $\DiffIdM$}, so it consists of diffeomophisms preserving $\func$ and isotopic to the identity, though the isotopy should not necessarily preserve $\func$.

Let $\func:\Mman\to\Pman$ be a smooth map, and $\Kman$ be a connected component of some \myemph{level set} $\func^{-1}(c)$ of $\func$, where $c \in\Pman$.
We will call $\Kman$ \myemph{regular}, whenever it does not contain critical points of $\func$.
Otherwise $\Kman$ will be called \myemph{critical}.

\begin{mydefinition}\label{def:MorseFunc}
A smooth map $f \in \Ci{\Mman}{\Pman}$ is called \myemph{Morse} if the following conditions are satisfied:
\begin{enumerate}[leftmargin=*, label={\rm(\arabic*)}]
\item $\func$ takes constant values on the connected components of the boundary $\partial\Mman$;
\item each critical point $\func$ is nondegenerate and is contained in $\Int{\Mman}$.
\end{enumerate}

A Morse map is
\begin{itemize}
\item \myemph{simple}, if every critical connected component of every critical level set of $\func$ contains a unique critical point;
\item \myemph{generic}, if each level set $\func$ contains no more than one critical point.
\end{itemize}
\end{mydefinition}
Figure~\ref{fig:morse_funcs} illustrates differences between simple and generic Morse maps.
For instance, each critical component $\Kman$ of each level set of a simple Morse map is either a point or a ``figure eight''.

\begin{figure}[h]
\centerline{
\begin{tabular}{ccccc}
\includegraphics[height=2cm]{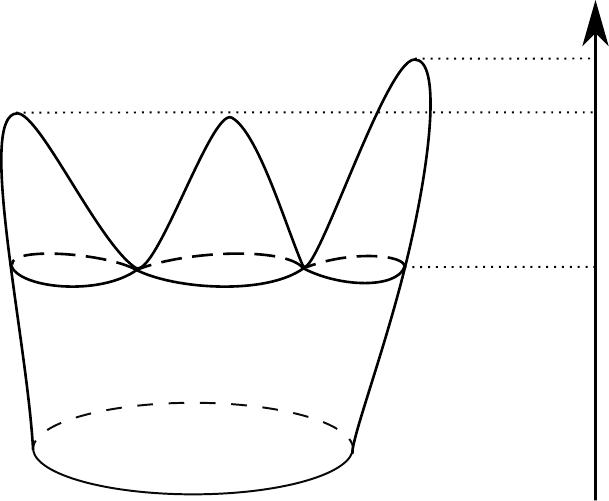} & \quad\qquad &
\includegraphics[height=2cm]{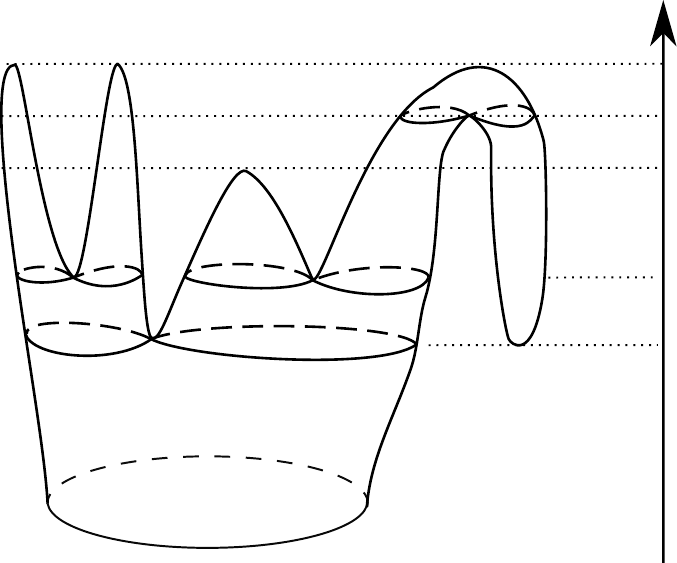} & \quad\qquad &
\includegraphics[height=2cm]{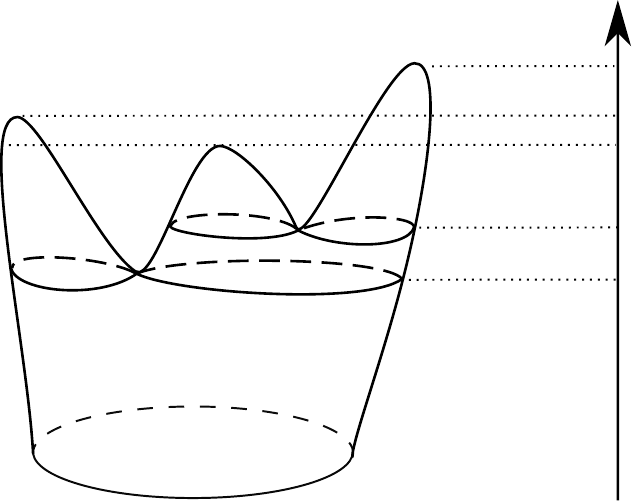} \\[2mm]
a) Morse function & &  b) simple Morse function & &  c) generic Morse function
\end{tabular}
}
\caption{}\label{fig:morse_funcs}
\end{figure}

The space of Morse maps $\Mman\to\Pman$ will be denoted by $\Morse{\Mman}{\Pman}$, and its subspaces consisting of simple, resp.\! generic, Morse maps will be denoted by $\MorseSmp{\Mman}{\Pman}$ and $\MorseGen{\Mman}{\Pman}$ respectively.
Obviously, the following inclusions hold:
\[
\MorseGen{\Mman}{\Pman} \ \subset \
\MorseSmp{\Mman}{\Pman} \ \subset \
\Morse{\Mman}{\Pman}.
\]

Notice that $\MorseGen{\Mman}{\Pman}$ is open and everywhere dense in the subset of $\Cinfty(\Mman,\Pman)$ taking constant values at boundary components of $\Mman$.

By~\cite{Sergeraert:ASENS:1972}, see also~\cite[Theorem~11.7]{Maksymenko:AGAG:2006}, for every $\func\in\Morse{\Mman}{\Pman}$ the map 
\begin{align}\label{equ:p:D_O_loc_triv_fibr}
p:\DiffM\to\Orbit{\func} &&
p(\dif) &= \eta(\dif,\func) = \func\circ\dif
\end{align}
is a locally trivial principal $\Stabilizer{\func}$-fibration.

In~\cite[Theorem~3.7]{Maksymenko:OsakaJM:2011} second author shown that for a wide class of maps on surfaces which contains all Morse maps the group $\StabilizerId{\func}$ is either contractible or homotopy equivalent to the circle, see also~\cite{Maksymenko:ProcIM:ENG:2010}, \cite{Maksymenko:UMZ:ENG:2012}.
He also proved in~\cite[Theorems 1.5, 1.6]{Maksymenko:AGAG:2006} that if $\func$ has at least one saddle point, then $\StabilizerId{\func}$ is contractible, $\pi_n \OrbitComp{\func}{\func} = \pi_n \Mman$ for $n\geq3$, $\pi_2 \OrbitComp{\func}{\func} = 0$, and for $\pi_1 \OrbitComp{\func}{\func}$ we have the following commutative diagram with exact rows:
\begin{equation}\label{equ:exact_seq_for_pi1Of}
\xymatrix{
1 \ar[r] & 
\pi_1\DiffIdM \ar[r] \ar@{^(->}[d]&
\pi_1\OrbitComp{\func}{\func} \ar[r]^{\partial} \ar@{=}[d] & 
\pi_0\StabilizerIsotId{\func} \ar[r] \ar[d]^{\rho} & 
1 \\
1 \ar[r]& 
\pi_1\DiffIdM  \oplus \bZ^k \ar[r] &
\pi_1\OrbitComp{\func}{\func} \ar[r]^{\lambda} & 
\Gf \ar[r] & 
1
}
\end{equation}
for some $k\geq0$ and a finite group $\Gf$ being \myemph{the group of automorphisms of the Kronrod-Reeb (or Lyapunov) graph $\KRGraphf$ of $\func$ induced by diffeomorphisms from $\StabilizerIsotId{\func}$}, see~\eqref{equ:Gf} below.
The upper row in~\eqref{equ:exact_seq_for_pi1Of} comes from the exact sequence of homotopy groups of the fibration~\eqref{equ:p:D_O_loc_triv_fibr}.
Moreover, if $\func$ is \myemph{generic}, then the group $\Gf$ is trivial (see Lemma~\ref{lm:f=sG} below), and $\OrbitComp{\func}{\func}$ is homotopy equivalent to one of the spaces
\begin{itemize}[label=$\bullet$]
\item $(S^1)^m=\underbrace{S^1\times\cdots\times S^1}_{m}$ if $\Mman\not=S^{2}, \bR{P}^2$;
\item $S^2$, if $\Mman=S^2$ and $\func$ have only two critical points maximum and minimum;
\item $\SO(3) \times (S^1)^m$, in all other cases,
\end{itemize}
for some $m\geq0$ which depends on $\func$.

The homotopy types of connected components of the spaces of Morse functions on compact spaces are calculated by E.~Kudryavtseva~\cite{Kudryavtseva:SpecMF:VMU:2012}, \cite{Kudryavtseva:MathNotes:2012}, \cite{Kudryavtseva:MatSb:2013}.
In those papers she also extended the above result to the case when $\Gf$ is non-trivial:
\begin{mytheorem}[\cite{Maksymenko:AGAG:2006}, \cite{Kudryavtseva:SpecMF:VMU:2012}, \cite{Kudryavtseva:MathNotes:2012}, \cite{Kudryavtseva:MatSb:2013}]
\label{th:hom_type_Of}
Let $\Mman$ be an orientable compact surface and $\func\in\Morse{\Mman}{\bR}$ be such that $\chi(\Mman) < |\Fix{\StabilizerIsotId{\func}}|$.
The last condition holds if $\chi(\Mman)<0$, or if $\func$ is generic and has at least one saddle critical point.
Then there exists a free action of the group $\Gf$ by isometries on an $m$-dimensional torus $(S^1)^m$ with flat metric, where $m=\rank(\pi_1\DiffId(\Mman) \ \oplus \ \bZ^k)$, such that $\OrbitComp{\func}{\func}$ is homotopy equivalent to the quotient space
\begin{itemize}[label=$\bullet$]
\item $(S^1)^m/\Gf$, if $\Mman\not=S^2$, and
\item $\SO(3) \ \times \ (S^1)^m/\Gf$, if $\Mman=S^2$.
\myqed
\end{itemize}
\end{mytheorem}

\begin{myremark}\rm
The group $\Gf$ in Theorem~\ref{th:hom_type_Of} is the \myemph{holonomy group of the compact affine manifold $(S^1)^m/\Gf$}.
Also recall that, e.g.~\cite{EarleEells:BAMS:1967}, \cite{Gramain:ASENS:1973},
\[
\pi_1\DiffId(\Mman) =
\begin{cases}
\bZ_2, & \text{$\Mman=S^2, \bR{P}^2$}, \\
\bZ\oplus\bZ, & \text{$\Mman=S^1\times S^1$ is a torus}, \\
\bZ, & \text{$\Mman=D^2, S^1\times[0,1]$, M\"obius band, or Klein bottle}, \\
0, & \text{for all other surfaces}.
\end{cases}
\]
Thus $\pi_1\DiffId(\Mman) \ \oplus \ \bZ^k$ is abelian and $m$ in Theorem~\ref{th:hom_type_Of} is its torsion-free rank.
\end{myremark}

The algebraic structure of the homomorphism $\rho:\pi_1\OrbitComp{\func}{\func}\to\Gf$ was described%
\footnote{In fact in~\cite{Maksymenko:DefFuncI:2014} this group $\Gf$ is denoted by $\mathbf{G}'(\func)$, but we removed here $'$ just to simplify the notations.}
in~\cite{Maksymenko:DefFuncI:2014} for the case when $\Mman$ is orientable and distinct from $S^2$ and $T^2 = S^1\times S^1$, (see the Lemma~\ref{lm:Gf_on_disk_cyl} below).
Furthermore, in a series of papers S.~Maksymenko and B.~Feshchenko~\cite{MaksymenkoFeshchenko:UMZ:ENG:2014}, \cite{Feshchenko:Zb:2015}, \cite{MaksymenkoFeshchenko:MFAT:2015}, \cite{MaksymenkoFeshchenko:MS:2015} computed $\pi_1\OrbitComp{\func}{\func}$ for the case $\Mman=T^2$.

Let
\begin{align*}
\ClassGfM    &= \{ \Gf \mid \func\in \Morse{\Mman}{\Pman}\}, \\
\ClassGfSmpM &= \{ \Gf \mid \func\in \MorseSmp{\Mman}{\Pman}\}, \\
\ClassGfGenM &= \{ \Gf \mid \func\in \MorseGen{\Mman}{\Pman} \}
\end{align*}
be the sets of isomorphisms classes of groups $\Gf$, where $\func$ runs over all Morse maps (respectively, simple, and generic Morse maps) from compact surface $\Mman$ to $\Pman$.

As noted above, $\ClassGfGenM = \{\UnitGroup\}$ consists of the trivial group $\UnitGroup$ only, see Lemma~\ref{lm:f=sG}.
In this paper, using the results from~\cite{Maksymenko:DefFuncI:2014}, we give an algebraic description of the sets $\ClassGfM$ and $\ClassGfSmpM$.

\begin{mytheorem}\label{th:ClassesGf}
Let $\PClassAll$ be the minimal family of isomorphism classes of groups satisfying the following conditions:
\begin{enumerate}[leftmargin=9ex, label={\rm(\alph*)}]
\item\label{enum:class:a1} $\UnitGroup\in \PClassAll$;	
\item\label{enum:class:b1} if $A, B\in \PClassAll$ and $n\in\bN$, then $A\times B$, $A\wr \bZ_{n} \, \in \, \PClassAll$.
\end{enumerate}
Let also $\PClassTwo \subset \PClassAll$ be the minimal subfamily with the property that
\begin{enumerate}[leftmargin=9ex,label={\rm(\alph*2)}]
\item\label{enum:class:a2} $\UnitGroup\in \PClassTwo$;	
\item\label{enum:class:b2} if $A, B\in \PClassTwo$, then $A\times B$, $A\wr \bZ_{2} \, \in \, \PClassTwo$.
\end{enumerate}
Then for each connected orientable compact surface $\Mman\not=S^2, T^2$ % the following sets coincide:
\begin{align}\label{equ:Classes_P_G}
\ClassGfM &= \PClassAll, &
\ClassGfSmpM &= \PClassTwo.
\end{align}
\end{mytheorem}

Here $A \wr\bZ_{n}$ is the \myemph{wreath} product of groups $A$ and $\bZ_{n}$, see~\S\ref{sect:wreath_products}.
In particular, Theorem~\ref{th:ClassesGf} shows that $\ClassGfM$ and $\ClassGfSmpM$ do not depend on $\Mman$ and $\Pman$, whenever $\Mman$ is orientable and distinct from $S^2$ and $T^2$.

\section{Preliminaries}\label{sect:preliminaries}
\paragraph{\bf Graph of a Morse map.}
Let $\func\in\Ci{\Mman}{\Pman}$, $\KRGraphf$ be a partition of $\Mman$ into connected components of level sets of $\func$, and $p:\Mman \to \KRGraphf$ be the canonical factor-map associating to each $x \in \Mman$ the connected component of the level set $\func^{-1}(\func(x))$ containing that point.

Endow $\KRGraphf$ with the \myemph{final} topology with respect to the mapping $p$, so a subset $A\subset \KRGraphf$ is open if and only if its inverse image $p^{-1}(A)$ is open in $\Mman$.
Then $\func$ induces a continuous map $\hat{\func}:\KRGraphf \to \Pman$, such that $\func=\hat{\func}\circ p$.

It is well known that if $\func$ is Morse, then $\KRGraphf$ has a structure of a one-dimensional CW-complex, and we will call it the \myemph{graph} of $\func$.
Its vertices correspond to the critical connected components of level sets of $\func$ and connected components of the boundary of the surface.

This graph was independently introduced in the papers by G.~Adelson-Welsky and A.~Kronrod~\cite{AdelsonWelskyKronrod:DANSSSR:1945}, and G.~Reeb~\cite{Reeb:CR:1946}, and often called \myemph{Kronrod-Reeb} or \myemph{Reeb} graph of $\func$.
It is a useful tool for understanding the topological structure of smooth functions on surfaces, e.g.~\cite{Kronrod:UMN:1950}, \cite{BolsinovFomenko:ENG:1997}.
It plays as well an important role in a theory of dynamical systems on manifolds and called \myemph{Lyapunov graph} of $\func$ following J.~Frank~\cite{Franks:Top:1985}, see also~\cite{RezendeFranzosa:TrAMS:1993}, \cite{Yu:TrAMS:2013}, \cite{RezendeLedesmaManzoli-NetoVago:TA:2018}.
The reason is that for generic Morse maps $\KRGraphf$ can be embedded into $\Mman$, so that $\func$ will be monotone on its edges, cf.~\ref{enum:lm:f=sG:rho_is_onto} of Lemma~\ref{lm:f=sG} below.

Denote by $\Homeo(\KRGraphf)$ the group of homeomorphisms of $\KRGraphf$.
Then each $\dif\in\Stabilizer{\func}$ leaves invariant every level set of $\func$, i.e. $\dif(\func^{-1}(c))=\func^{-1}(c)$ for all $c\in\Pman$.
Hence it induces a homeomorphism $\rho(\dif)$ of the graph $\KRGraphf$ making commutative the following diagram:
\begin{equation}\label{equ:2x2_M_Graph}
\xymatrix{
\Mman \ar[rr]^-{p} \ar[d]_-{\dif} &&
\KRGraphf \ar[rr]^-{\hat{\func}} \ar[d]^-{\rho(\dif)} &&
\Pman \ar@{=}[d]  \\
\Mman \ar[rr]^-{p} &&
\KRGraphf \ar[rr]^-{\hat{\func}} &&
\Pman
}
\end{equation}
so that the correspondence $h\mapsto g(h)$ is a homomorphism $\rho :\Stabilizer{\func} \to \Homeo(\KRGraphf)$.

One easily checks, see Lemma~\ref{lm:f=sG} below, that the image $\rho(\Stabilizer{\func})$ is a finite subgroup of $\Homeo(\KRGraphf)$.
Thus, in a certain sense, $\rho(\Stabilizer{\func})$ encodes ``discrete symmetries of $\func$''.

It also follows from E.~Kudryavtseva and A.~Fomenko~\cite{KudryavtsevaFomenko:VMU:2013} that for every finite group $G$ there exists a Morse function $\func$ on some compact surface $\Mman$ such that $G\cong\rho(\Stabilizer{\func})$ and $\func$ takes the value $i$ at each critical point of index $i$, for $i=0,1,2$.

We will be interested in the image of the group $\StabilizerIsotId{\func}:=\Stabilizer{\func}\cap\DiffIdM$, which we denote by $\Gf$:
\begin{equation}\label{equ:Gf}
\Gf := \rho(\StabilizerIsotId{\func}).
\end{equation}

The following statement collects several properties of homomorphism $\rho$, e.g.\! in~\cite{Maksymenko:AGAG:2006}, \cite{Kudryavtseva:MathNotes:2012}, \cite{Kudryavtseva:MatSb:2013}.
We include it for the convenience of the reader and since they will be used in the proof of Theorem~\ref{th:ClassesGf}. 
\begin{mylemma}\label{lm:f=sG}
\begin{enumerate}[leftmargin=*, label={\rm\arabic*)}]
\item\label{enum:lm:f=sG:image_rho_is_finite}
The image $\rho(\Stabilizer{\func})$ is a finite subgroup of $\Homeo(\KRGraphf)$.
\item\label{enum:lm:f=sG:rho_is_onto}
There exists a continuous map $s:\KRGraphf \to \Mman$ being ``right homotopy'' inverse to the factor-map $p:\Mman\to\KRGraphf$, i.e. $p\circ s$ is homotopic to the identity map $\id_{\KRGraphf}$.
If $\func$ is simple, then $s$ can be chosen to be an embedding such that $p\circ s = \id_{\KRGraphf}$.
Hence the homomorphism $p_1:H_1(\Mman,\bZ) \to H_1(\KRGraph,\bZ)$ induced by the projection map $p$ is surjective, cf.~\cite[Corollary~3.8]{RezendeLedesmaManzoli-NetoVago:TA:2018}.

\item\label{enum:lm:f=sG:a:rho_h_id_on_H1}
Let $\dif\in\StabilizerIsotId{\func}$, so $\rho(\dif) \in\Gf$.
Homeomorphism $\rho(\dif)$ induces the identity automorphism of the first homology group $H_1(\KRGraphf,\bZ)$.
In particular, $\rho(\dif)$ leaves invariant each simple cycle of $\KRGraphf$.
\item\label{enum:lm:f=sG:b:rho_h_id_for_gener_f} 
If $\func$ is generic and $\dif\in\StabilizerIsotId{\func}$, then $\rho(\dif)$ is the identity automorphism of $\KRGraphf$.
In particular, the group $\Gf$ is trivial, whence $\ClassGfGenM=\{ \UnitGroup \}$.
\end{enumerate}
\end{mylemma}
\begin{proof}
\ref{enum:lm:f=sG:image_rho_is_finite}
Notice that $\hat{\func}$ is \myemph{monotone} on edges of $\KRGraphf$ and $\rho(\dif)$ preserves $\hat{\func}$.
Hence if $\dif_1, \dif_2 \in \Stabilizer{\func}$ are such that the homeomorphisms $\rho(\dif_1)$ and $\rho(\dif_2)$ induce the same permutation of edges of $\KRGraphf$, then $\rho(\dif_1)=\rho(\dif_2)$.
In other words, each element of $\rho(\Stabilizer{\func})$ is determined by its permutation of edges.
Since $\KRGraphf$ has only finitely many edges, it follows that $\rho(\Stabilizer{\func})$ is finite.

\ref{enum:lm:f=sG:rho_is_onto}
For every vertex $v$ of $\KRGraphf$ let $\Kman_{v}=p^{-1}(v)$ be the corresponding critical component of some level set of $\func$ or the boundary component of $\partial\Mman$.
Fix any point $z_{v}\in\Kman_{v}$ and set $s(v) = z_{v}$.

\begin{figure}[h]
\centerline{
\begin{tabular}{ccc}
\includegraphics[height=2cm]{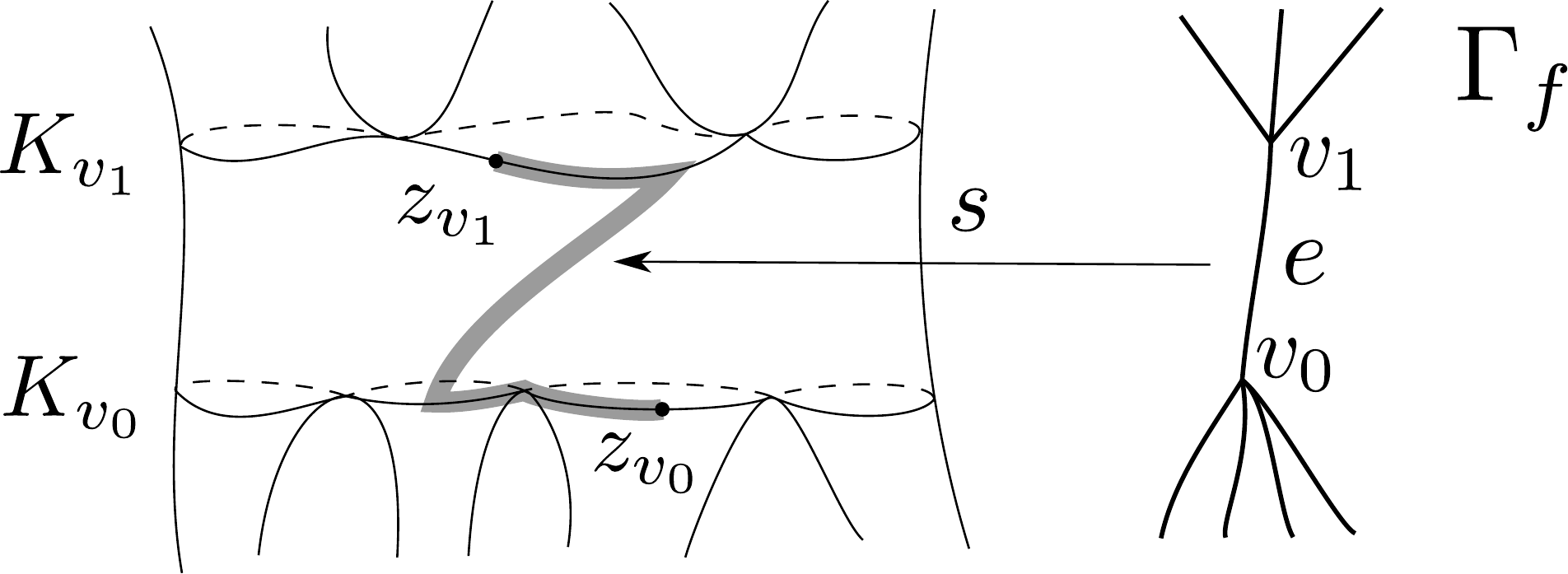} & \quad\qquad &
\includegraphics[height=2cm]{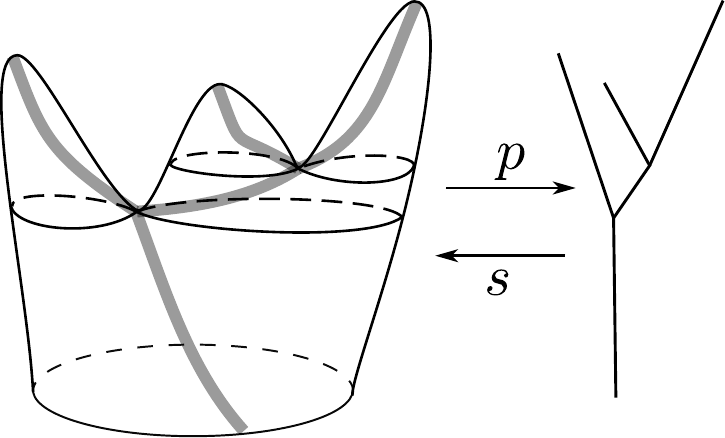} \\[2mm]
a) && b)
\end{tabular}
}
\caption{}\label{fig:section}
\end{figure}

We need to extend $s$ to the edges of $\KRGraphf$.
Let $e$ be an open edge of $\KRGraphf$ with vertices $v_0$ and $v_1$ (which may coincide if $\Pman=S^1$, so $\func$ is a circle-valued map) and $\phi_{e}:[0,1]\to\KRGraphf$ be the \myemph{characteristic} map for $e$, so $\phi_{e}(i)=v_i$, $i=0,1$, and $\phi$ homeomorphically maps $(0,1)$ onto $e$.

Fix an arbitrary embedding $\psi_{e}:[0.2, 0.8] \to\Mman$ such that $\psi_{e}(0.2) \in \Kman_{v_0}$, $\psi(0.8) \in \Kman_{v_1}$, and $\psi_{e}$ also homeomorphically maps $(0.2, 0.8)$ onto $e$.
Extend $\psi_{e}$ to the map $[0,1]\to\Mman$ so that $\psi_{e}|_{[0,0.2]}: [0,0.2] \to \Kman_{v_0}$ is a path connecting $z_{v_0}$ with $\psi(0.2)$, while $\psi_{e}|_{[0.8,1]}: [0.8,1] \to \Kman_{v_1}$ is a path connecting $\psi_{e}(0.8)$ with $z_{v_1}$.
Finally, define $s|_{e} = \psi_{e}\circ\phi_{e}^{-1}:e\to\Mman$, see Figure~\ref{fig:section}a).

Then $s:\KRGraphf\to\Mman$ is continuous, $p\circ s(v)=v$ for each vertex $v\in\KRGraphf$, and $p\circ s(\overline{e})=\overline{e}$ for each edge $e$ of $\KRGraphf$.
This implies that $p\circ s$ is homotopic to $\id_{\KRGraphf}$ relatively to the set of vertices.

Anther proof of~\ref{enum:lm:f=sG:rho_is_onto} is presented in~\cite[Corollary~3.8]{RezendeLedesmaManzoli-NetoVago:TA:2018}.

Also notice that if $\func$ is simple, and $\Kman_{v}$ is a critical component of some level set of $\func$, then one can choose $z_{v}$ to be a unique critical point belonging to $\Kman_{v}$.
Moreover, for every edge $e$ we can define $\psi_{e}:[0,1]\to\Mman$ to be an embedding.
In this case we will have that $p\circ s = \id_{\KRGraphf}$, see Figure~\ref{fig:section}b), where the image of $s$ is shown in gray.

\ref{enum:lm:f=sG:a:rho_h_id_on_H1}
Notice that for every $\dif\in\Stabilizer{\func}$ \eqref{equ:2x2_M_Graph} induces the following commutative diagram:
\begin{equation}\label{equ:H1_2x2_M_Graph}
\xymatrix{
H_1(\Mman,\bZ) \ar[r]^{p_1} \ar[d]_{\dif_1} & H_1(\KRGraphf,\bZ) \ar[d]^{\rho(\dif)_1} \\
H_1(\Mman,\bZ) \ar[r]^{p_1}               & H_1(\KRGraphf,\bZ)
}
\end{equation}
in which horizontal arrows are surjective due to~\ref{enum:lm:f=sG:rho_is_onto}.

If $\dif\in\StabilizerIsotId{\func}=\Stabilizer{\func}\cap\DiffIdM$, then $\dif_1$ is the identity, whence so is $\rho(\dif)_1$.

\ref{enum:lm:f=sG:b:rho_h_id_for_gener_f} 
Suppose $\func$ is generic.
First, we prove that \myemph{$\rho(\dif)$ fixes each vertex of $\KRGraphf$}, i.e. each of the following sets is invariant with respect to $\dif$:
\begin{itemize}
\item
critical components of $\Kman_1, \ldots,\Kman_a$ of all sets of the level $\func$, and
\item
connected components $\Bman_1, \ldots,\Bman_b$ of the boundary $\partial\Mman$.
\end{itemize}

Indeed, since $\dif(\partial\Mman)=\partial\Mman$ and $\dif$ is isotopic to the identity map $\id_{\Mman}$, we get that $\dif(\Bman_j)=\Bman_j$ for all $j=1,\ldots,b$.

Moreover, by assumption $\func\circ\dif=\func$, whence $\dif(\func^{-1}(c)) = \func^{-1}(c)$ for all $c\in\Pman$.
Denote $c_i = \func(\Kman_i)$, $i=1,\ldots,a$.
Then $\Kman_i$ and $\dif(\Kman_i)$ are critical components of the same level set $\func^{-1}(c_i)$.
But $\func$ is generic, so  $\Kman_i$ is a unique critical component of $\func^{-1}(c_i)$.
Hence $\dif(\Kman_i) = \Kman_i$ for all $i$.
Thus every vertex of $\KRGraphf$ is fixed under $\rho(\dif)$.

Furthermore, due to \ref{enum:lm:f=sG:a:rho_h_id_on_H1} $\rho(\dif)$ trivially act on the group $H_1(\KRGraphf,\bZ)$ and therefore it leaves invariant each edge of $\KRGraphf$ with its orientation.
But as noted in~\ref{enum:lm:f=sG:image_rho_is_finite} $\rho(\dif)$ is determined by its permutation of edges, whence $\rho(\dif)$ is the identity.
\myqed
\end{proof}

\paragraph{\bf Morse equality.}
Let $\func\in\Morse{\Mman}{\Pman}$.
Denote by $c_i(\func)$ the number of critical points $\func$ of index $i=0,1,2$.
Then it is well known, e.g.~\cite[Corollary~2.10]{Nicolaescu:MorseTh:2007}, that $\Mman$ admits a CW-structure such that there is a bijection between the cells of dimension $i$ and critical point of $\func$ of index $i$.
In particular, we have the following identity, see~\cite[Corollary~2.15]{Nicolaescu:MorseTh:2007}:
\begin{equation}\label{equ:morse_eq}
\chi(\Mman) = c_0(\func) - c_1(\func) + c_2(\func).
\end{equation}

\section{Wreath products}\label{sect:wreath_products}
For a set $\bSet$ denote by $\Sigma(\bSet)$ the permutation group on $\bSet$, i.e.\! the group of all self-bijections of $\bSet$ with respect to the composition of maps.
Then $\Sigma(\bSet)$ naturally acts on $\bSet$ from the left.
Assume that some $\bGroup$ group also acts from the left of $\bSet$, which gives a certain homomorphism $\actHom:\bGroup\to\Sigma(\bSet)$.

Let $\aGroup$ be another group, and $\Map(\bSet,\aGroup)$ be the group of \myemph{all} maps $\bSet\to\aGroup$ with respect to the pointwise multiplication.
Then $\bGroup$ acts from the right on $\Map(\bSet,\aGroup)$ by the rule: if $\alpha\in\Map(\bSet,\aGroup)$ and $\hh\in\bGroup$, then the result of the action of $\hh$ on $\alpha$ is just the composition of maps:
\[
\alpha\circ\actHom(\hh):\bSet \xrightarrow{~\actHom(\hh)~} \bSet \xrightarrow{~\alpha~} \aGroup,
\]
which will be denoted by $\alpha^{\hh}$.
Then the semidirect product 
\[
\aGroup\wr_{\actHom}\bGroup:= \Map(\bSet,\aGroup)\rtimes_{\actHom}\bGroup
\]
corresponding to this $\bGroup$-action will be called the \myemph{wreath product of $\aGroup$ and $\bGroup$ over $\actHom$}.
In this case $\bGroup$ is also called the \myemph{top} group of $\aGroup\wr_{\actHom}\bGroup$.

Thus, $\aGroup\wr_{\actHom}\bGroup$ is the Cartesian product of sets $\Map(\bSet,\aGroup)\times\bGroup$ with the multiplication defined by the following formula:
\begin{equation}\label{equ:prod_in_wr}
(\alpha_1,\hh_1) \, (\alpha_2,\hh_2) =\bigl(\alpha_1^{\hh_2} \alpha_2, \,\hh_1\circ\hh_2\bigr).
\end{equation}
for $(\alpha_1,\hh_1), (\alpha_2,\hh_2)\in\Map(\bSet,\aGroup)\times\bGroup$.

If $\bGroup$ is a subgroup of $\Sigma(\bSet)$ and $j:\bGroup\subset\Sigma(\bSet)$ is the inclusion map, then the corresponding wreath product $\aGroup\wr_{j}\bGroup$ will be denoted by $\aGroup\wr_{\bSet}\bGroup$:
\begin{equation}\label{equ:action_of_perm_subgroups}
\aGroup\wr_{\bSet}\bGroup:= \aGroup\mathop{\wr}\limits_{j:\bGroup\subset\Sigma(\bSet)}\bGroup.
\end{equation}

Finally notice that the group $\bGroup$ acts on itself by left shifts, whence we have a natural inclusion map $j:\bGroup\to\Sigma(\bGroup)$ and can define the wreath product $\aGroup\wr_{j}\bGroup$, which in this particular case is usually denoted by $\aGroup\wr\bGroup$.

For example, the group $\aGroup\wr \bZ_{n}$ (appearing in Theorem~\ref{th:ClassesGf}) is the Cartesian product of sets $\underbrace{\aGroup\times \cdots \times \aGroup}_{n} \times \bZ_n$ with the following multiplication:
\[
(a_0,a_1,\, \cdots,a_{n-1}; k) \, (b_0,b_1,\, \cdots,b_{n-1}; l)  =
(a_l b_0, a_{l+1} b_1,\, \cdots,a_{l-1} b_{n-1}; k+l),
\]
where all indexes are taken modulo $n$.

\begin{mydefinition}\label{def:class_RG}
Let $\GG = \{G_i\}_{i\in \Lambda}$ an arbitrary family of finite groups.
Denote by $\RR(\GG)$, the minimal family of isomorphism classes of groups satisfying the following conditions:
\begin{enumerate}[label={\rm(\roman*)}]
\item $\UnitGroup \in \RR(\GG)$;
\item if $A,B\in\RR(\GG)$ and $G\in\GG$ then $A\times B, \, A \wr G \, \in \, \RR(\GG)$.
\end{enumerate}
\end{mydefinition}

In particular, in Theorem~\ref{th:ClassesGf}
\begin{align*}
\PClassAll &= \RR(\{ \bZ_{n}\}_{n\in\bN}), &
\PClassTwo &= \RR(\{ \bZ_{2}\} ).
\end{align*}

Consider examples of groups which belong to these families, and groups that do not belong to them.

\newcommand\msep{ && }
\begin{myexample}
For arbitrary $a,b,n,m\in\bN$:
\begin{align*}
&\UnitGroup,  \msep  \bZ_{n},  \msep  \bZ_{m}\wr \bZ_{n},  \msep  (\bZ_{m})^{a}\times (\bZ_{n})^{b},  \msep \bigl((\bZ_{m} \wr \bZ_{a})\times\bZ_{b}\bigr) \wr \bZ_{n}  \ \in  \ \PClassAll, \\
%%%%%%%%%%%
&\UnitGroup,  \msep \bZ_{2},  \msep \bZ_{2}\wr \bZ_{2},  \msep (\bZ_{2})^{a}\times (\bZ_{2})^{b},  \msep \bigl((\bZ_{2} \wr \bZ_{2})\times\bZ_{2}\bigr) \wr \bZ_{2}  \ \in  \ \PClassTwo.
\end{align*}
\end{myexample}

\begin{myremark}\rm
Wreath products $D\wr \bZ_{p}$, where $p$ is a simple, and $D$ is a $p$-group, for example, the groups $\bZ_{p}\wr \bZ_{p}$, $(\bZ_{p}\wr \bZ_{p})\wr \bZ_{p}$, were studied in~\cite{IovanovMasonMontgomery:MRL:2014}, \cite[Theorems~4.7, 5.5]{Keilberg:arXiv:2018} in connection with the so-called FSZ-property.
\end{myremark}

\begin{mylemma}\label{lm:wreath_prod_properties}{\rm \cite{Neumann:MathZ:1964}}
Let $A$ and $B$ be nontrivial groups.
Then
\begin{enumerate}[leftmargin=*, label={\rm(\arabic*)}]
\item\label{enum:wreath_prop_1}
$A \wr B \cong P \times Q$ for some non-trivial groups $P$ and $Q$ if and only if $B$ is finite, say of order $n$, and $A$ has a nontrivial direct abelian factor $C$ (which can be identical with $A$) such that every $c\in C$ has a unique root of $n$-th degree in $C$, {\rm\cite[Theorem~7.1]{Neumann:MathZ:1964}};

\item\label{enum:wreath_prop_2}
$A \wr B \cong P \wr Q$ if and only if $P\cong A$ and $Q\cong B$, {\rm\cite[Theorem~10.3]{Neumann:MathZ:1964}}.
\myqed
\end{enumerate}
\end{mylemma}

\begin{myexample}\label{ex:Z2_wr_Z2Z4_notin_GG}
For any $q\geq 2$ and $r\geq 1$ the group $W = \bZ_q \wr (\bZ_q\times\bZ_{qr})$ does not belong to $\PClassAll$.
\end{myexample}
\begin{proof}
Suppose, $W\in\PClassAll$.
Then either $W = P\times Q$ for some nontrivial groups $P,Q\in\PClassAll$, or $W = P\wr\bZ_{n}$ for some non-trivial group $P\in\PClassAll$ and $n\geq2$.

Suppose $W = \bZ_q \wr (\bZ_q\times\bZ_{qr}) \cong P \times Q$.
Then, due to~\ref{enum:wreath_prop_1} of Lemma~\ref{lm:wreath_prod_properties}, each element of the \myemph{nontrivial group} $\bZ_{q}$ must have a unique root of order $|\bZ_q\times\bZ_{qr}|=q^2r$.
However, $q^2r \cdot x =0$ for all $x\in\bZ_q$, and therefore every nonzero element of $\bZ_q$ has no roots of order $q^2r$ in $\bZ_{q}$.
Thus, $W$ can not be decomposed into a nontrivial direct product.

Suppose that $W = \bZ_q \wr (\bZ_q\times\bZ_{qr}) \cong P\wr \bZ_{n}$ for some $n\geq2$.
Then, according to~\ref{enum:wreath_prop_2} of Lemma~\ref{lm:wreath_prod_properties}, $P\cong \bZ_q$ and $\bZ_{n}\cong\bZ_q\times\bZ_{qr}$.
But $\bZ_q\times\bZ_{qr}$ is not a cyclic group, since $q$ divides $qr$.
Hence $W$ is not a wreath product with the top group $\bZ_n$. 

Thus, both cases are impossible, and so $W\not\in\PClassAll$.
\myqed
\end{proof}

\paragraph{Jordan's theorem.}
We will show now that the families of groups $\PClassAll$ and $\PClassTwo$ are related with Jordan's theorem on the group of automorphisms of trees.

For $n\geq1$ let $X_n := \{1,\ldots,n\}$ and $\Sigma_n:=\Sigma(X_n)$ be the permutation group of $X_n$, so $\Sigma_{n}$ acts from the left of $X_n$.
Let $\id_{\Sigma_n}:\Sigma_n\to\Sigma_n$ be the identity isomorphism of $\Sigma_n$, and $j_n:\Sigma_n \to \Sigma(\Sigma_n)$ the inclusion corresponding to the actions of $\Sigma_n$ on itself by left shifts.
Then for every group $\aGroup$ one can define two wreath products:
\begin{align*}
\aGroup\wr_{X_n} \Sigma_n &:= \Map(X_n,\aGroup)\rtimes_{\id_{n}}\Sigma_n, &
\aGroup\wr \Sigma_n &:= \Map(\Sigma_n,\aGroup)\rtimes_{j_{n}}\Sigma_n,
\end{align*}
corresponding to the actions of $\Sigma_n$ on $X_n$ and on itself respectively.
It is obvious, that $\aGroup\wr_{X_1} \Sigma_1 \cong \aGroup\wr \Sigma_1 \cong \aGroup$.
On the other hand, for $n\geq2$ these wreath products, in general, are not isomorphic.
Indeed, if $\aGroup$ is a finite group, then 
\begin{align*}
|\aGroup\wr_{X_n} \Sigma_n| &= |\aGroup|^n n!\,, &
|\aGroup\wr \Sigma_n| &= |\aGroup|^{n!} n!\,.
\end{align*}

Let $\RR_{\Sigma}$ be the minimal family of isomorphism classes of groups satisfying the following conditions:
\begin{enumerate}[label={\rm(\roman*)}]
\item the trivial group $\UnitGroup$ belongs to $\RR_{\Sigma}$;
\item if $A,B\in\RR_{\Sigma}$ and $n\geq 1$, then $A\times B, \, A \wr_{X_n} \Sigma_n \, \in \, \RR_{\Sigma}$.
\end{enumerate}

Thus, $\RR_{\Sigma}$ differs from $\RR(\{\Sigma_n\}_{n\geq1})$ by the assumption that the wreath products $A\wr \Sigma_n$ are replaced with $A\wr_{X_n} \Sigma_n$.
As noted in~\cite[Nr.~54]{Polya:AM:1937} the following theorem is a reformulation following the old result of C.~Jordan~\cite{Jordan:JRAM:1869} characterizing automorphisms groups of finite trees, see also~\cite[Proposition~1.15]{Babai:HC:1994} and \cite[Lemma~5]{KrasikovSchonheim:DM:1985}.

\begin{mytheorem}\label{th:Jordan_theorem}{\rm (C.~Jordan, \cite{Jordan:JRAM:1869}).}
A group $G$ is isomorphic to the group of all automorphisms of a finite tree if and only if $G\in\RR_{\Sigma}$.
\myqed
\end{mytheorem}

Note that Theorem~\ref{th:ClassesGf} characterizes in a similar way families of groups $\Gf$ of automorphisms of graphs $\KRGraphf$ of Morse maps $\func:\Mman\to\Pman$ for compact surfaces $\Mman\not=S^2, T^2$ induced by diffeomorphisms belonging to $\StabilizerIsotId{\func}$.
For instance, one can assume that $\Mman$ is a $2$-disk, whence $\KRGraphf$ is a tree, see~\ref{enum:lm:f=sG:a:rho_h_id_on_H1} of Lemma~\ref{lm:f=sG},
whence such groups correspond to special automorphisms of finite trees.

\section{Proof of Theorem~\ref{th:ClassesGf}}
First we will give another description of the sets $\PClassAll$ and $\PClassTwo$.
Consider the following countable alphabet:
\[
 \mathcal{A}=\{ \ 1 \ , \ (\ , \ )\ ,  \ \times\ ,  \ \wr \bZ_{2}\ , \ \wr \bZ_{3}\ , \ \wr \bZ_{4}\ , \ \cdots \ \},
\]
that is, for each $n\in\bN$ we regard $\wr\bZ_{n}$ as one letter.

Denote by $\AdmisWords$ the minimal subset consisting of finite words from the alphabet $\mathcal{A}$ satisfying the following conditions:
\begin{enumerate}[label={\rm(\roman*)}]
\item $1 \in \AdmisWords$;
\item if $\aword, \bword\in \AdmisWords$ and $n\in\bN$, then the words $(\aword)$, $(\aword) \times (\bword)$, and $(\aword)\wr\bZ_{n} \ \in \AdmisWords$.
\end{enumerate}

The words from $\AdmisWords$ will be called \myemph{admissible}.
Evidently, every admissible word $\word$ uniquely defines (up to an isomorphism) a certain finite group $\nu(\word)$.
Let $\PClassAll^{*}$ the family of isomorphism classes of groups that are written by admissible words.
We claim that $\PClassAll^{*} = \PClassAll:=\RR(\{ \bZ_{n}\}_{n\in\bN})$, whence the correspondence $\word\mapsto\nu(\word)$ will be a surjective map $\nu: \AdmisWords\to\PClassAll$.

Indeed, $\PClassAll^{*}$ obviously satisfies the conditions~\ref{enum:class:a1} and~\ref{enum:class:b1} of the definition of $\PClassAll$ in Theorem~\ref{th:ClassesGf} and therefore contains the minimal set $\PClassAll$, that is $\PClassAll^{*} \supset \PClassAll$.

Conversely, if $\PClassAll'$ is an arbitrary family of isomorphism classes of groups satisfying the same conditions, then it necessarily contains all groups written by admissible words, that is $\PClassAll^{*} \subset \PClassAll'$.
In particular, $\PClassAll^{*} \subset \PClassAll$, and hence, $\PClassAll^{*} = \PClassAll$.

\medskip

Consider further the following finite subalphabet in $\mathcal{A}$:
\begin{equation*}\label{equ:alphabet_2}
 \mathcal{A}_2=\{ \ 1 \ , \ ( \ , \ )\ ,  \ \times\ ,  \ \wr \bZ_{2} \ \}
\end{equation*}
and denote by $\AdmisWords_2$ is the minimal subset of finite words from the alphabet $\mathcal{A}_2$, satisfying the following conditions:
\begin{enumerate}[label={\rm(\roman*)}]
\item $1 \in \AdmisWords_2$;	
\item if $A, B \in \AdmisWords_2$, then the words $(A)$, $(A) \times (B)$, $(A) \wr \bZ_{2}$ also belong to $\AdmisWords_2$.
\end{enumerate}
Then $\AdmisWords_2 \subset \AdmisWords$ and by similar arguments, $\PClassTwo = \nu(\AdmisWords_2)$.

Let us also mention the following simple statement which we will use several times:
\begin{mylemma}\label{lm:subwords}
A word $\word\in\AdmisWords$ belongs to $\AdmisWords_2$ if and only if $\word$ does not contain letters $\wr\bZ_n$ for $n\geq3$.
Hence if $\word\in\AdmisWords_2$ and $\aword$ is a subword (a consecutive sequence of symbols) of $\word$ belonging to $\AdmisWords$, then $\aword\in\AdmisWords_2$ as well.
\myqed
\end{mylemma}

Suppose now that $\Mman$ is a compact orientable surface distinct from $S^2$ and $T^2$.
First we will prove that
\begin{align}\label{equ:P_in_G}
\PClassTwo &\subset \ClassGfSmpM, &
\PClassAll &\subset \ClassGfM.
\end{align}
In other words, we need to show that for every admissible word $\word \in \AdmisWords$ there exists a Morse map $\func \in \Morse{\Mman}{\Pman}$ such that $\nu(\word)\simeq\Gf$.
Moreover, if $\word \in \AdmisWords_2$, then such an $\func$ can be chosen to be a simple Morse map.

We will use an induction on the length $\len(\word)$ of the word $\word$ simultaneously for all orientable surfaces $\Mman\not=S^2, T^2$ and $\Pman=\bR$ or $S^1$.

a) Let $\len(\word)=1$, so $\word=1$, and hence $\nu(\word)\cong\UnitGroup$ is the trivial group.
Then, due to Lemma~\ref{lm:f=sG}, for each generic Morse map $\func:\Mman\to\Pman$ on any compact surface $\Mman$ we have that 
\[\nu(\word) \cong \UnitGroup \cong \Gf\in \ClassGfSmpM \subset \ClassGfM.\]

b) Let $\len(\word) = n\geq1$ and assume that we have already proven that
\begin{itemize}%[leftmargin=3ex]
\item for each admissible word $\aword \in \AdmisWords$ of length $\len(\aword)<n$ and each orientable surface $\Mman\not=S^2,T^2$ and $\Pman=\bR$ or $S^1$ the group $\nu(\aword) \in \ClassGfM$;
\item moreover, if $\aword \in \AdmisWords_2$, then $\nu(\aword) \in \ClassGfSmpM$.
\end{itemize}

We need to prove that $\nu(\word) \in \ClassGfM$, and moreover, if $\word \in \AdmisWords_2$ then $\nu(\word) \in \ClassGfSmpM$.

Fix any generic Morse map $\gfunc:\Mman\to\Pman$ having a unique local maximum $z\in\Mman$.
According to the definition of admissible words we should consider the following three cases.

\begin{enumerate}[leftmargin=*, itemindent=\parindent, label={\rm(\arabic*)}]
\item\label{enum:M:brackets}
Suppose $\word = (\aword)$ for some $\aword \in \AdmisWords$.
Then $\nu(\word)\cong \nu(\aword)$ and
\begin{align*}
\len(\aword) &= n-2<n.
\end{align*}
Hence, by induction, $\nu(\aword) \in \ClassGfM$.
Also if $\word \in \AdmisWords_2$, then by Lemma~\ref{lm:subwords} $\aword \in \AdmisWords_2$ as well.
Hence, by induction, $\nu(\word)\cong \nu(\aword) \in \ClassGfSmpM$.

\item\label{enum:M:prod}
Suppose $\word=(\aword) \times (\bword)$ for some admissible words $\aword, \bword \in \AdmisWords$.
Then their lengths are less than $n$, whence, by the induction, there exist \myemph{functions} $\alpha, \beta \in \Morse{D^2}{\bR}$ on a $2$-disk such that $\nu(\aword)\cong \Ga$ and $\nu(\bword)\cong \Gb$.
Using $\alpha$ and $\beta$ we can change $\gfunc$ in a neighborhood of the point $z$ such that the resulting map $\func$ will belong to $\Morse{\Mman}{\Pman}$ and $\Gf\cong \Ga \times \Gb$, see Figure~\ref{fig:AxB}.
Then,
\[
 \nu(\aword) \times \nu(\bword) \cong \Ga \times \Gb \cong \Gf \ \in \ \ClassGfM.
\]

If $\word \in \AdmisWords_2$, then due to Lemma~\ref{lm:subwords} $\aword, \bword \in \AdmisWords_2$ as well, so by induction, $\alpha$ and $\beta$ can be chosen to be simple.
Then $\func$ will also be simple.
Hence $\nu(\aword) \times \nu(\bword) \cong \Gf\in\ClassGfSmpM$.

\begin{figure}[h]
\centerline{
\begin{tabular}{ccccc}
\includegraphics[height=1.8cm]{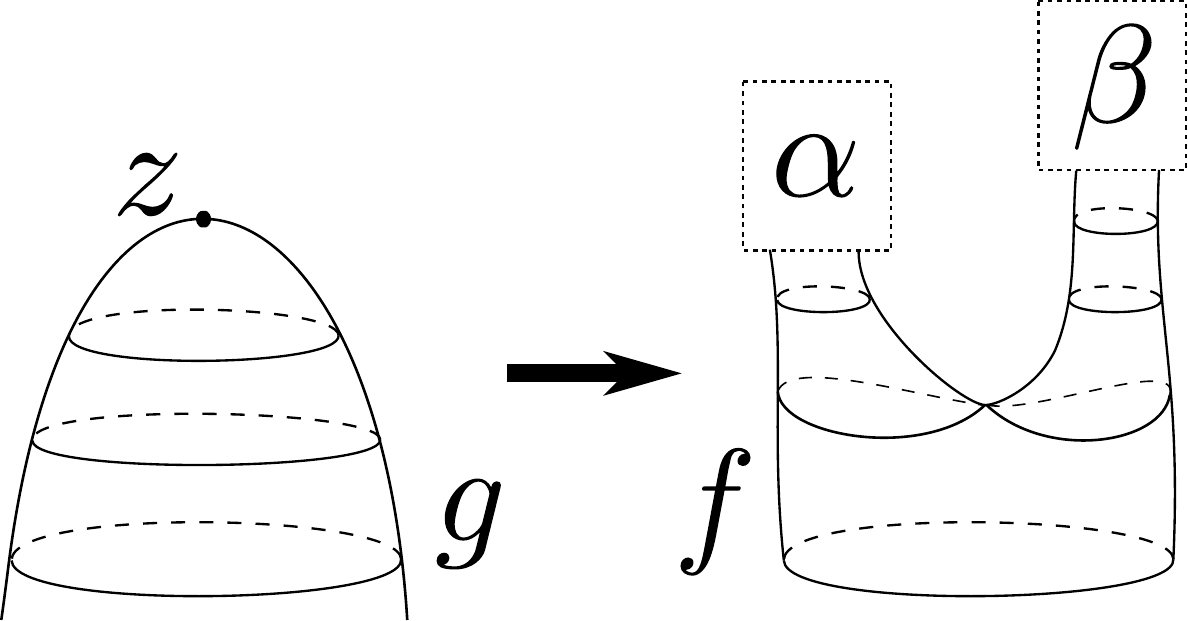} & \quad\qquad &
\includegraphics[height=1.8cm]{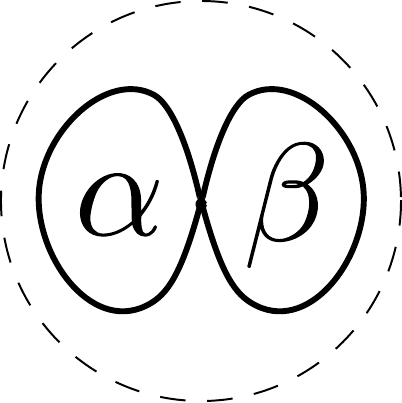} & \quad\qquad &
\includegraphics[height=1.8cm]{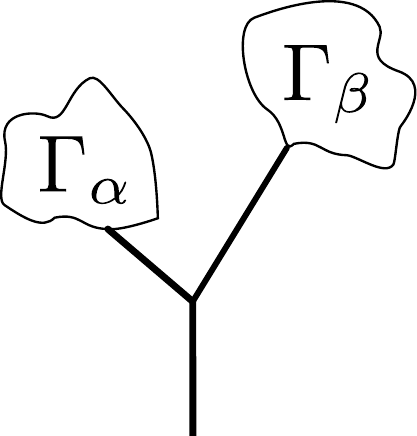} \\[2mm]
a) Functions $\gfunc$ and $\func$ & &  b) New critical component of & &  c) Graph $\KRGraphf$ \\
                               & &        level set of $\func$
\end{tabular}
}
\caption{$\Gf\cong \Ga \times \Gb$}\label{fig:AxB}
\end{figure}

\item\label{enum:M:wreathk}
In all other cases, $\word=(\aword) \wr \bZ_{k}$, for some $k\in\bN$ and some admissible word $A \in \AdmisWords$ such that $\len(\aword) = n-3 < n$.
Then by induction there exists a Morse \myemph{function} $\alpha \in \Morse{D^2}{\bR}$ on a $2$-disk, such that $\nu(A)\cong\Ga$.
Now, using $\alpha$ we can change $\gfunc$ in a neighborhood of the point $z$ such that the resulting map $\func$ will belong to $\Morse{\Mman}{\Pman}$, and $\Gf\cong \Ga\wr \bZ_{k}$, see Figure~\ref{fig:AwrZk} for $k=5$.
Then,
\[
\nu(\word) \cong  \nu(\aword) \wr \bZ_{2} \cong \Ga  \wr \bZ_{k} \cong \Gf \ \in \ \ClassGfM.
\]

\begin{figure}[h]
\centerline{
\begin{tabular}{ccc}
\includegraphics[height=2cm]{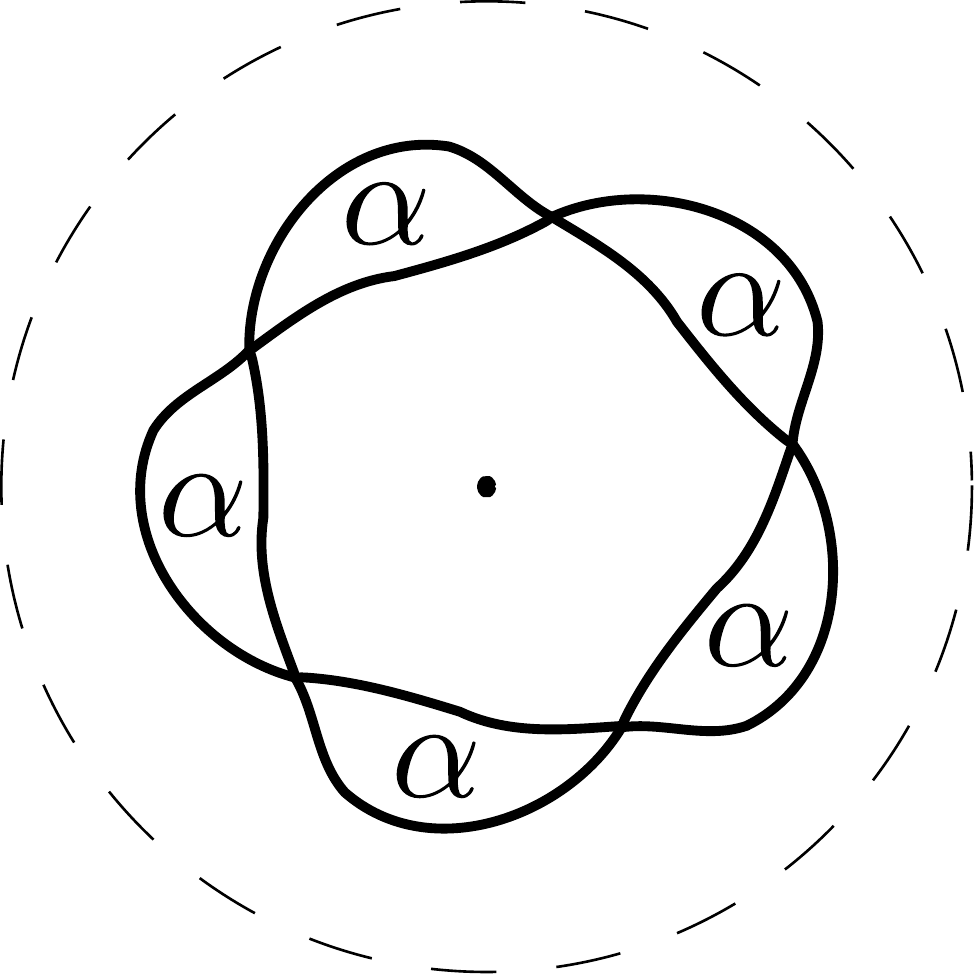} & \qquad\qquad &
\includegraphics[height=2cm]{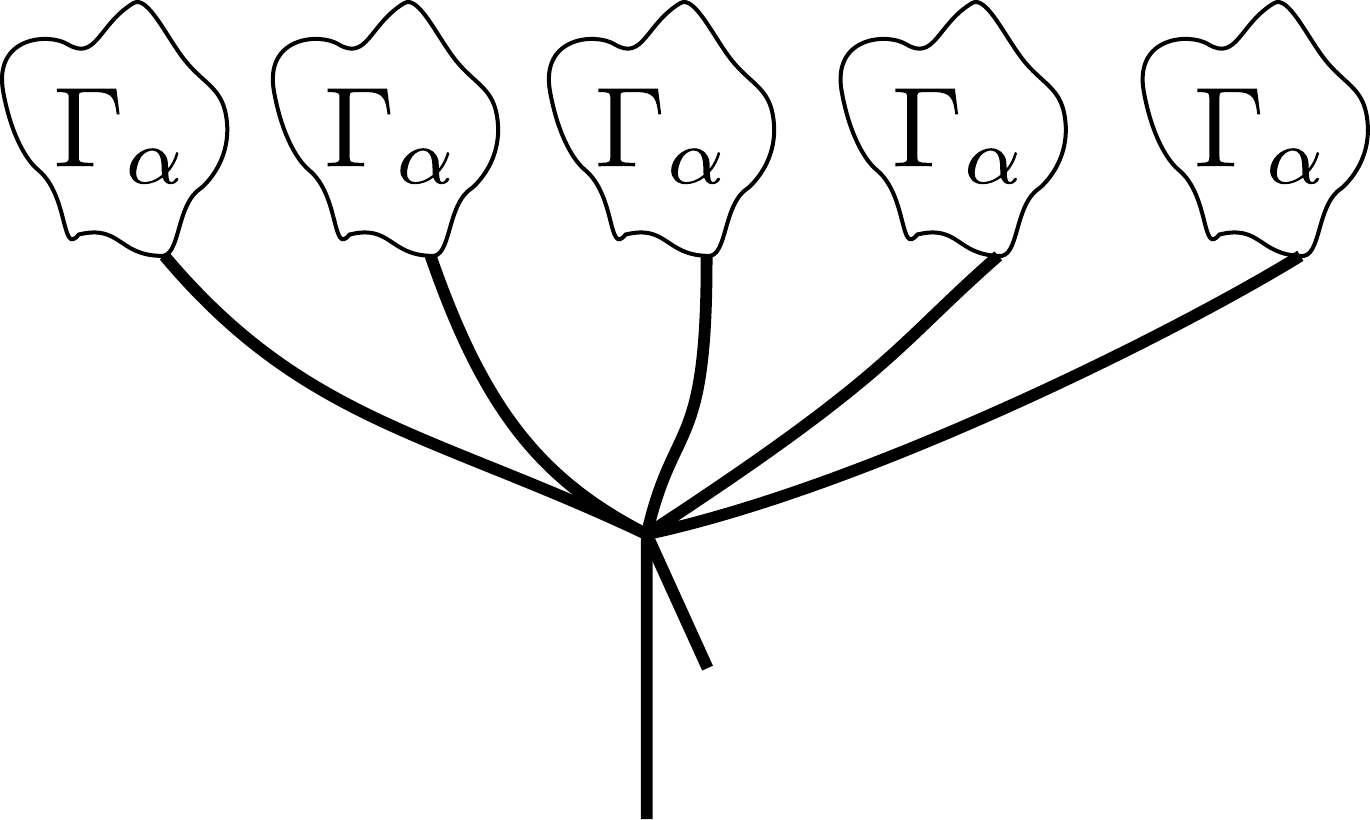} \\
a) New critical component of & & b) Graph $\KRGraphf$ \\
   level set of $\func$
\end{tabular}
}
\caption{$\Gf\cong \Ga \wr \bZ_{5}$}\label{fig:AwrZk}
\end{figure}

If $\word\in\AdmisWords_2$, then $k=2$ and by Lemma~\ref{lm:subwords} $\aword \in \AdmisWords_2$.
Therefore, by induction, $\alpha$ can be chosen to be simple, and since $k=2$, the map $\func$ will also be simple.
Thus $\nu(\word) \cong \nu(\aword) \wr \bZ_{2}\cong \Gf\in\ClassGfSmpM$, see Figure~\ref{fig:AwrZ2}.

\begin{figure}[h]
\centerline{
\begin{tabular}{ccccc}
\includegraphics[height=1.8cm]{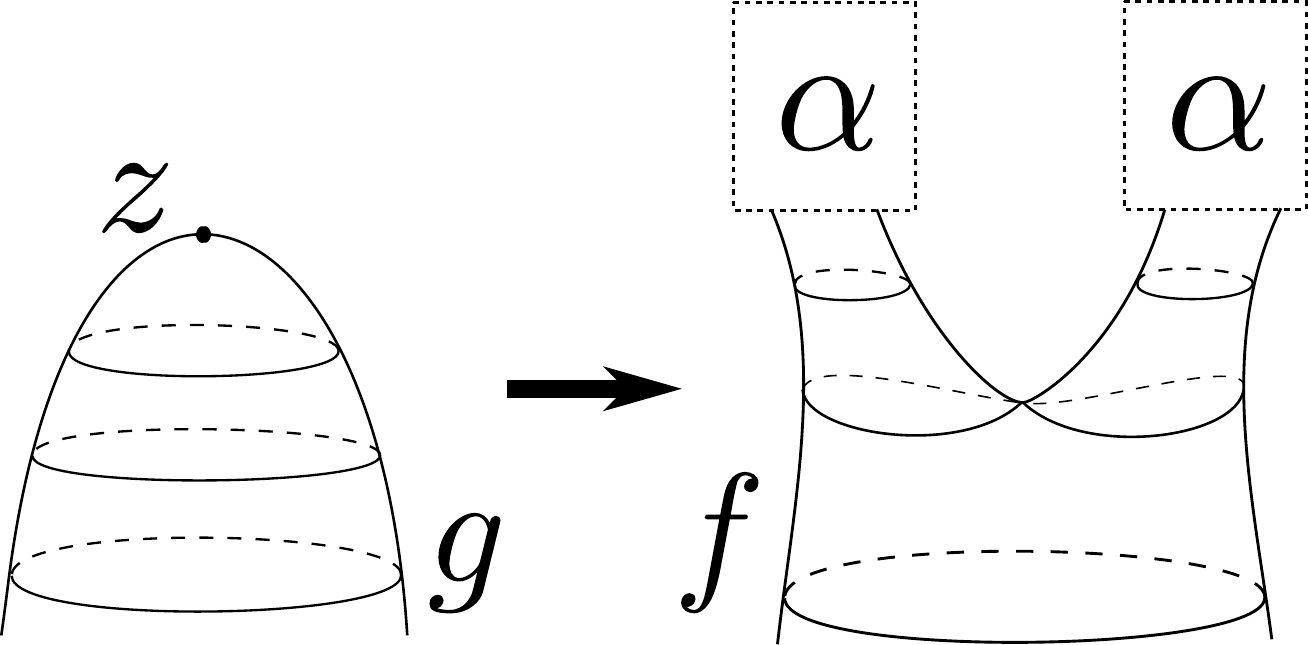} & \quad &
\includegraphics[height=1.8cm]{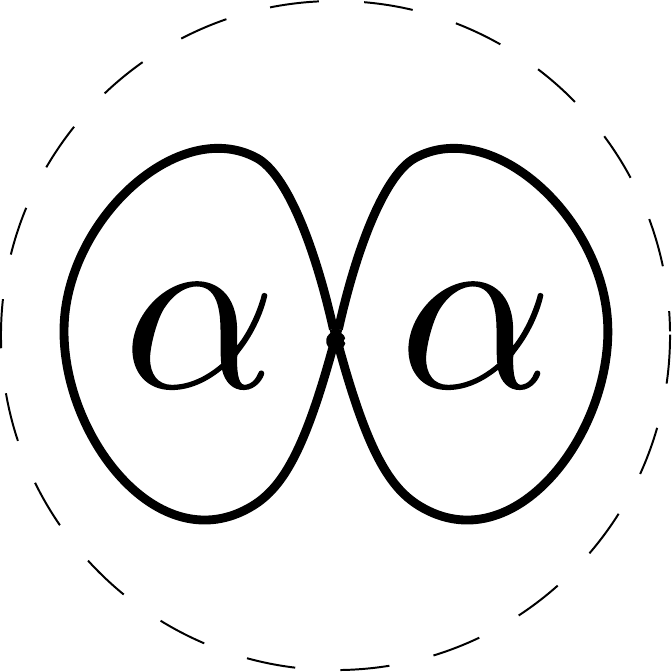} & \quad &
\includegraphics[height=1.8cm]{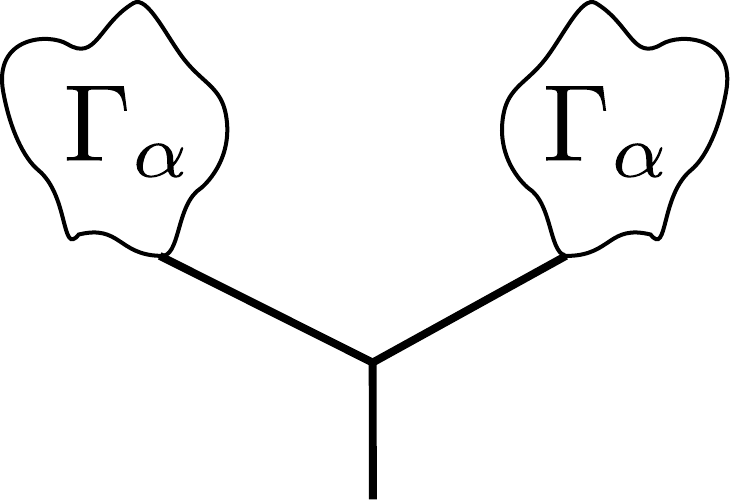} \\[2mm]
a) Functions $\gfunc$ and $\func$  & & b) New critical component & &  c) Graph $\KRGraphf$ \\
                                   & &    line levels $\func$
\end{tabular}
}
\caption{$\Gf\cong \Ga \wr \bZ_{2}$}\label{fig:AwrZ2}
\end{figure}
\end{enumerate}

This proves~\eqref{equ:P_in_G}.
Let us establish the inverse inclusions
\begin{align}\label{equ:G_in_P}
\ClassGfSmpM &\subset \PClassTwo, &
\ClassGfM &\subset \PClassAll.
\end{align}

First notice that it is enough to prove them only for $\Mman=D^2, S^1\times [0,1]$.
Indeed, recall the following statement:
\begin{mylemma}\label{lm:Gf_prod_decomp}
{\rm\cite[Theorem~1.7]{Maksymenko:MFAT:2010}, \cite[Theorem~5.3]{Maksymenko:DefFuncI:2014}.}
Let $\Mman$ be a compact orientable surface distinct from $S^2$ and $T^2$ and let $\func\in\Morse{\Mman}{\Pman}$.
Then there exist finitely many subsurfaces $\Mman_1,\ldots,\Mman_k \subset \Mman$ such that
\begin{enumerate}[label={\rm(\arabic*)}]
\item each $\Mman_i$ is either $2$-disk $D^2$, or a cylinder $S^1\times[0,1]$, and $\Mman_i\cap \Mman_j = \varnothing$ for $i\not= j$;
\item the restriction $\func|_{\Mman_i}: \Mman_i\to\Pman$ belongs to $\Morse{\Mman_i}{\Pman}$, $i=1,\ldots,k$;
\item\label{enum:lm:Gf_prod_decomp:prod} $\Gf \cong \Grpf{\func|_{\Mman_1}} \ \times \ \cdots \ \times \Grpf{\func|_{\Mman_k}}$.
\myqed
\end{enumerate}
\end{mylemma}
Now if~\eqref{equ:G_in_P} is shown for a $2$-disk and a cylinder, then $\Grpf{\func|_{\Mman_i}} \in \PClassAll$ for all $i=1,\ldots,k$.
Since $\PClassAll$ is closed with respect to finite direct products, we will get that $\Gf\in\PClassAll$.
Moreover, if $\func$ is a simple, then $\func|_{\Mman_i}$ will also be simple.
Therefore $\Grpf{\func|_{\Mman_i}} \in \PClassTwo$ for all $i$, whence, $\Gf\in\PClassTwo$ as well.
This will prove the inclusions~\eqref{equ:G_in_P}.

\medskip

Let $\Mman=D^2$ or $S^1\times I$ and $\func\in \Morse{\Mman}{\Pman}$.
We need to show that $\Gf\in\PClassAll$, and if $\func$ is simple then $\Gf\in\PClassTwo$.

Without loss of generality we can assume that $\func$ takes a local minimum on some connected component $\gamma$ of the boundary $\partial\Mman$.
We will use an induction on the number $c_1(\func)$ of saddle critical points of $\func$.

\medskip

1) Suppose $c_1(\func)=0$.
If $\Mman=S^1\times[0,1]$ then by~\eqref{equ:morse_eq}
\[
0 = \chi(S^1\times[0,1]) = c_{0}(\func) + c_{2}(\func),
\]
whence $c_{0}(\func) = c_{2}(\func)=0$, so $\func$ have no critical points at all, see Figure~\ref{fig:no_crpt}a).

If $\Mman=D^2$, then again by~\eqref{equ:morse_eq}
\[1 = \chi(D^2) = c_{0}(\func) + c_{2}(\func),\]
and therefore $\func$ has only one critical point which is an extremum (in this case the maximum, since we assumed that $\func$ takes the minimum value on $\partial D^2 = \gamma$), see Figure~\ref{fig:no_crpt}b).
In both cases the graph $\KRGraphf$ of the map $\func$ consists of one edge, whence
\[\Gf\cong\UnitGroup\in\PClassTwo \subset \PClassAll.\]

\begin{figure}[ht]
\centerline{
\begin{tabular}{ccc}
\includegraphics[height=2cm]{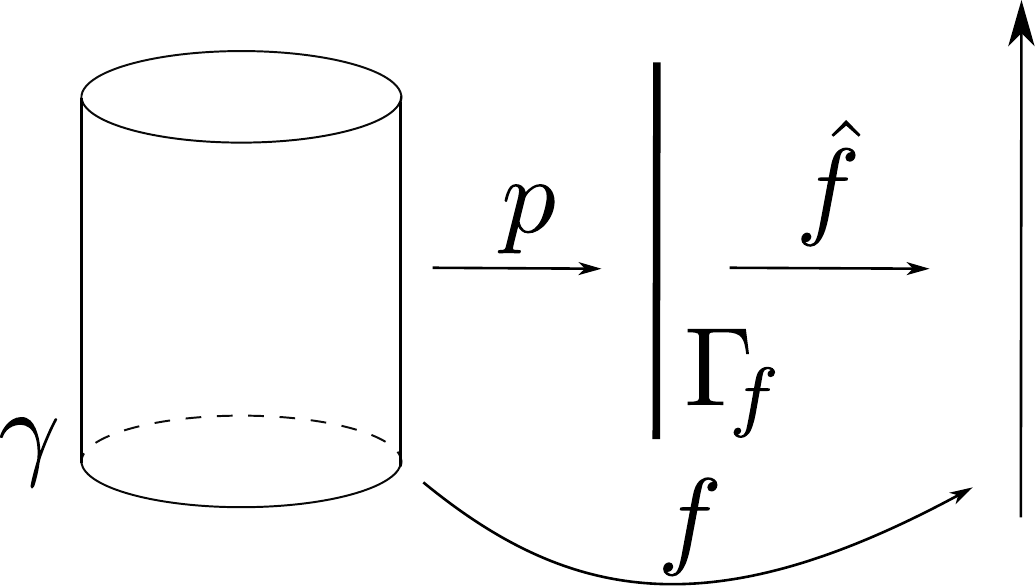} & \qquad\qquad &
\includegraphics[height=2cm]{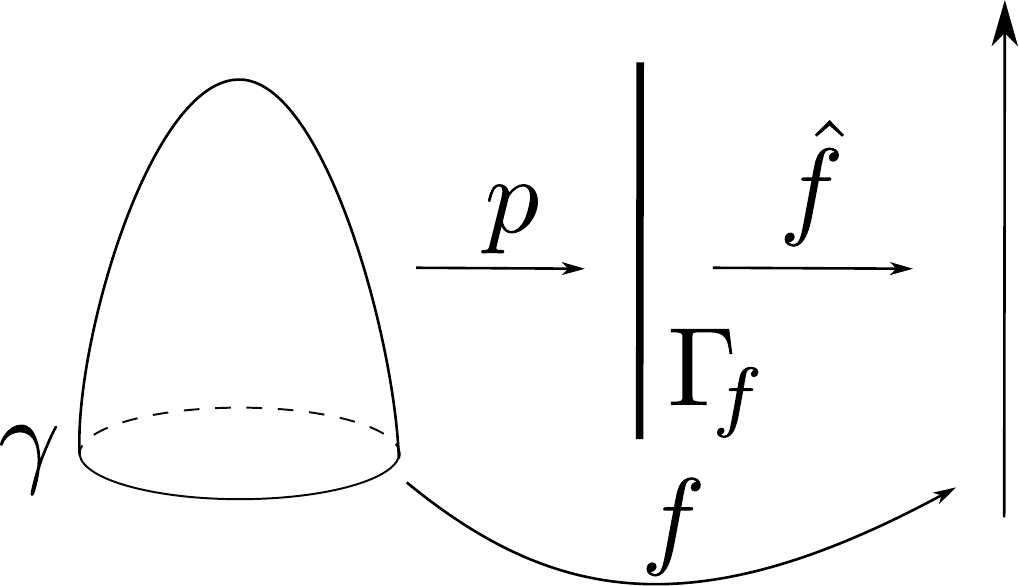} \\[2mm]
a) && b)
\end{tabular}}
\caption{Functions with minimal number of critical points}\label{fig:no_crpt}
\end{figure}

2) Now let $c_1(\func)\geq1$, i.e. $\func$ has at least one saddle critical point.
Assume that we have proved that $\Ga \in \PClassAll$ for all Morse maps $\alpha$ on $D^2$ and $S^1\times I$ with $c_1(\alpha) < c_1(\func)$.

Let $u$ be the vertex of the graph $\KRGraphf$ corresponding to the boundary component $\gamma \subset\partial\Mman$, $(u,v)$ be a unique edge $\KRGraphf$ incident to $u$, $\Kman=p^{-1}(v)$ be the critical component of level set of $\func$ corresponding to $v$, and $c = \func(\Kman)$.

Fix small $\varepsilon>0$ and let $\Rman$ be the connected component of $\func^{-1}([c-\varepsilon,c+\varepsilon])$ containing $\Kman$, (since $\Pman$ is an abelian group for both cases $\bR$ and $S^1=\bR/\bZ$, the segment $[c-\varepsilon,c+\varepsilon] \subset\Pman$ is well defined).
Decreasing $\varepsilon$ one can also assume that $\Kman$ does not intersect $\partial\Mman$ and contains no other critical points of $\func$ except for the one belonging to $\Kman$, see Figure~\ref{fig:regnbh}.

Then $\dif(\Rman)=\Rman$ for all $\dif\in\StabilizerIsotId{\func}$, and therefore $\overline{\Mman\setminus\Rman}$ is also invariant with respect to $\StabilizerIsotId{\func}$.
Let
\begin{itemize}[label=$\bullet$]
\item
$\Xman_0,\ldots,\Xman_a$ be all connected components of $\overline{\Mman\setminus\Rman}$, invariant with respect to $\StabilizerIsotId{\func}$, that is $\dif(\Xman_i)=\Xman_i$ for all $\dif\in\StabilizerIsotId{\func}$ and $i=0,\ldots,a$,

\item
$\Yman_1,\ldots,\Yman_b$ be all other connected components of $\overline{\Mman\setminus\Rman}$.
\end{itemize}
We can also assume that $\gamma\subset \Xman_0$;

Notice that each connected component of the boundary $\partial\Mman$ is preserved by each $\dif\in\StabilizerIsotId{\func} = \Stabilizer{\func} \cap \DiffIdM$, since it is isotopic to $\id_{\Mman}$.
Moreover, as $\gamma\subset \Xman_0$, and every connected component of $\partial\Rman$ (being a simple closed curve) separates $\Mman$, it follows that
\begin{itemize}[label=$\bullet$]
\item
$\Xman_0$ is a cylinder, $\Xman_i$, $i=1,\ldots,a$, is either a $2$-disk, or a cylinder, and every $\Yman_j$, $j=1,\ldots,b$, is a $2$-disk;

\item
the restriction of $\func$ to each $\Xman_i$ and $\Yman_j$ is Morse in the sense of Definition~\ref{def:MorseFunc}, and $c_1(\func|_{\Xman_i}), \, c_1(\func|_{\Yman_j})  <  c_1(\func)$, so, by induction, $\Grpf{\func|_{\Xman_i}}, \, \Grpf{\func|_{\Yman_j}}  \in \PClassAll$.
\end{itemize}
\begin{mylemma}\label{lm:Gf_on_disk_cyl}
{\rm\cite[Theorem~5.7]{Maksymenko:DefFuncI:2014}}
\begin{enumerate}[leftmargin=*, label={\rm(\arabic*)}]
\item\label{enum:lm:Gf_on_disk_cyl:b0}
Suppose that each connected component of $\overline{\Mman\setminus\Rman}$ is invariant with respect to $\Stabilizer{\func}$, i.e. there is no $\Yman_j$-components.
Then $\Gf \cong \prod_{i=1}^{a} \Grpf{\func|_{\Xman_i}}$.

\item\label{enum:lm:Gf_on_disk_cyl:b_pos}
Otherwise, $b=cm$ for some $c,m\in\bN$, and the surfaces $\{\Yman_j\}$ can be renumbered by pairs of indices $(k,l)$ as follows:
\begin{equation}\label{equ:renumbering_Yij}
\begin{matrix}
\Yman_{1,0} & \Yman_{1,1} & \cdots & \Yman_{1,m-1} \\
\Yman_{2,0} & \Yman_{2,1} & \cdots & \Yman_{2,m-1} \\
\cdots      & \cdots      & \cdots & \cdots        \\
\Yman_{c,0} & \Yman_{c,1} & \cdots & \Yman_{c,m-1}
\end{matrix}
\end{equation}
so that every diffeomorphism $\dif\in \Stabilizer{\func}$ \myemph{cyclically shifts the columns in~\eqref{equ:renumbering_Yij}}, i.e.\! there exists $n\in\{0,\ldots,m-1\}$ depending on $\dif$, such that
\[\dif(\Yman_{k,l}) = \Yman_{k,\, l+n\, \mathrm{mod} \, m}\]
for all $k,l$.
Moreover
\begin{itemize}[label=$\bullet$]
\item if $a=0$, then $\Gf \cong \left( \prod_{k=1}^{c} \Grpf{\func|_{\Yman_{k,0}}} \right) \wr \bZ_{m}$;
\item otherwise $a=1$ and $\Gf \cong \Grpf{\func|_{\Xman_{1}}} \times  \left( \Bigl( \prod_{k=1}^{c} \Grpf{\func|_{\Yman_{k,0}}}\Bigr) \wr \bZ_{m}\right)$.
\myqed
\end{itemize}
\end{enumerate}
\end{mylemma}
\begin{figure}[ht]
\begin{tabular}{ccc}
\includegraphics[height=2.3cm]{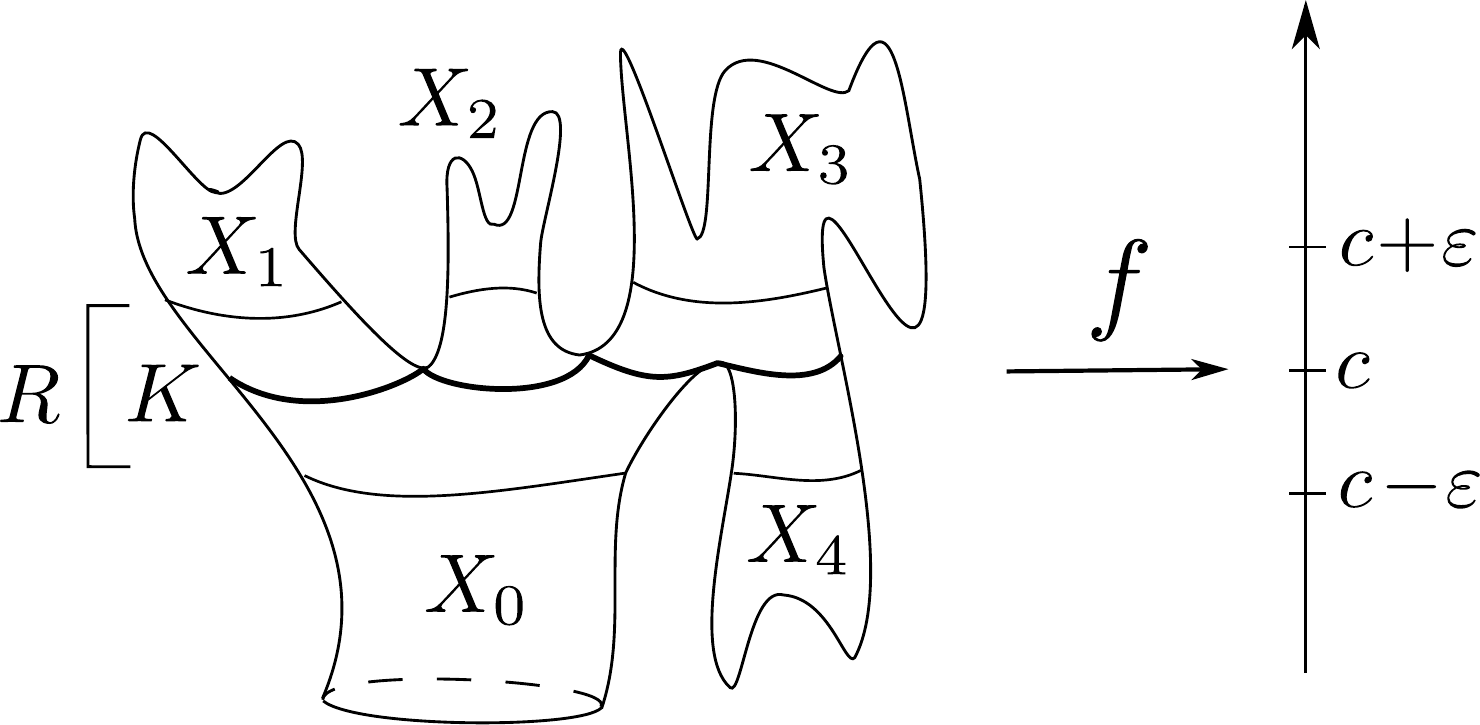} & \qquad\qquad &
\includegraphics[height=2.3cm]{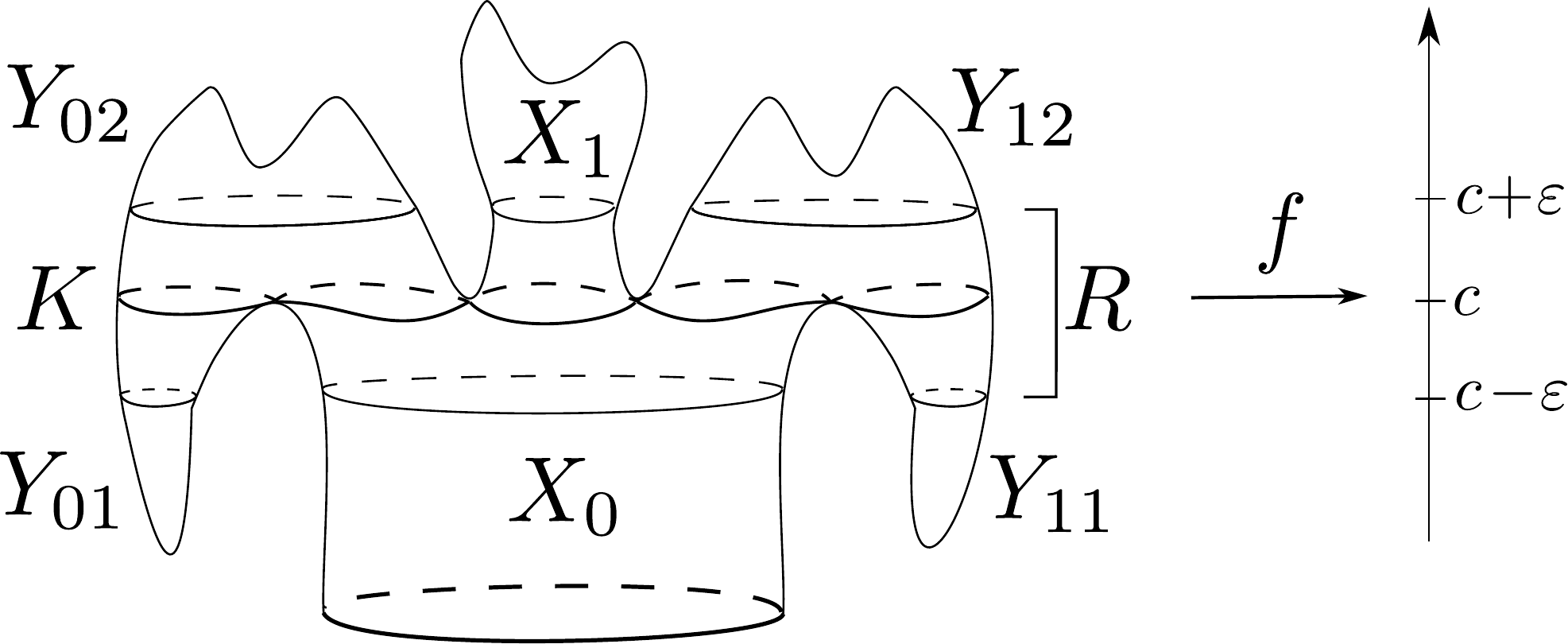} \\[2mm]
Case (1) && Case (2)
\end{tabular}
\caption{}\label{fig:regnbh}
\end{figure}

Since the family $\PClassAll$ is closed with respect to direct product and wreath products with finite cyclic groups from the top, we get from Lemma~\ref{lm:Gf_on_disk_cyl} that $\Gf\in\PClassAll$.

Assume now that $\func$ is simple.
Then $\overline{\Mman\setminus\Rman}$ consists of three connected components $\Xman_0$, $\Zman_1$, $\Zman_2$, such that $\Xman_0$ is a cylinder containing $\partial D^2$, and $\Zman_1$, $\Zman_2$ are discs, see Figure~\ref{fig:AwrZ2:details}.
Moreover, $\func|_{\Zman_1}$ and $\func|_{\Zman_2}$ are simple Morse maps and therefore, by induction, $\Grpf{\func|_{\Zman_1}}$, $\Grpf{\func|_{\Zman_2}}\in\PClassTwo$.

\begin{figure}[ht]
\centerline{\includegraphics[height=2.3cm]{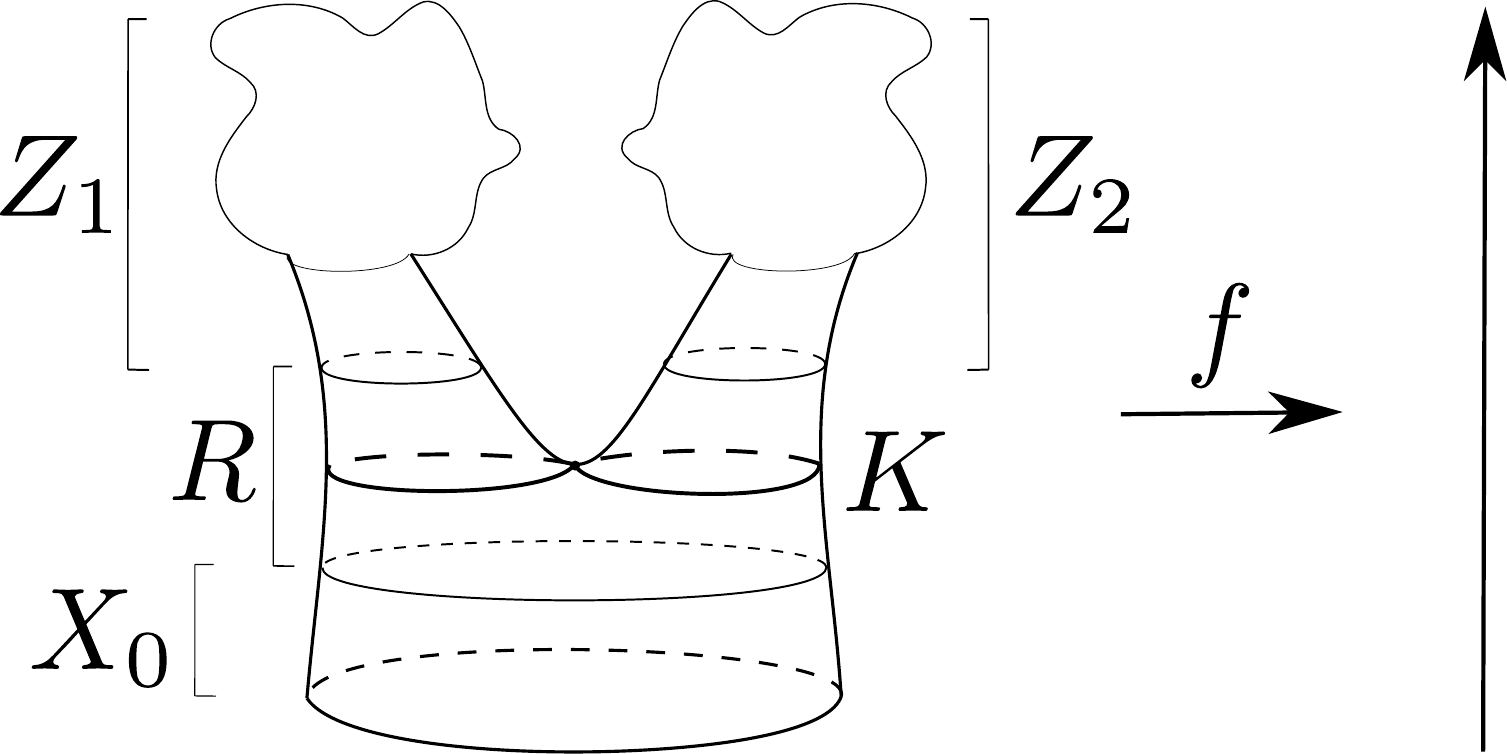}}
\caption{}\label{fig:AwrZ2:details}
\end{figure}

If $\Zman_1$ and $\Zman_2$ are invariant with respect to $\StabilizerIsotId{\func}$, then by the case~\ref{enum:lm:Gf_on_disk_cyl:b0} of Lemma~\ref{lm:Gf_on_disk_cyl}, \[\Gf \cong \Grpf{\func|_{\Zman_1}} \times \Grpf{\func|_{\Zman_2}} \ \in \ \PClassTwo.\]

Otherwise, there exists $\dif\in\Stabilizer{\func}$ such that $\dif(\Zman_1)=\Zman_2$.
Therefore, we have the case~\ref{enum:lm:Gf_on_disk_cyl:b_pos} of Lemma~\ref{lm:Gf_on_disk_cyl}.
Moreover, the numbering in~\eqref{equ:renumbering_Yij} can be done as follows:
\begin{align*}
\Yman_{1,0} &= \Zman_1, &
\Yman_{1,1} &= \Zman_2,
\end{align*}
that is $c=1$, $m=2$.
Therefore $\Gf \cong \Grpf{\func|_{\Yman_{1,0}}} \wr \bZ_{2} \in \PClassTwo$.
Theorem~\ref{th:ClassesGf} is completed.
\myqed

\section{Generalizations}
In fact, in~\cite{Maksymenko:DefFuncI:2014} it was studied the set $\mathcal{F}(\Mman,\Pman)$ of smooth maps $\func:\Mman\to\Pman$ taking constant values at connected components of $\partial\Mman$ and being smoothly equivalent to \myemph{homogeneous polynomials without multiple factors} in neighborhoods of their critical points.
As the polynomials $\pm x^2\pm y^2$ are homogeneous and have no multiple factors, it follows from Morse lemma that $\Morse{\Mman}{\Pman}\subset\mathcal{F}(\Mman,\Pman)$.
We will sketch here a variant of Theorem~\ref{th:ClassesGf} for the maps from $\mathcal{F}(\Mman,\Pman)$.

One could check that every $\func\in\mathcal{F}(\Mman,\Pman)$ has only isolated critical points, and it is possible to define its graph $\KRGraphf$ and the homomorphism $\rho:\Stabilizer{\func}\to\Aut(\KRGraphf)$ similarly to~\S\ref{sect:preliminaries}.

Let also $\FolStabilizer{\func}$ be the subgroup of $\ker(\rho)$ consisting of diffeomorphisms $\dif$ such that for every \myemph{degenerate local extreme} $z$ of $\func$ the induced tangent map $T_z\dif: T_z\Mman\to T_z\Mman$ is the identity.
Denote $\FolStabilizerIsotId{\func} = \FolStabilizer{\func} \cap\DiffIdM$ and put
\begin{equation}\label{equ:Gf_1}
\Gf =  \StabilizerIsotId{\func} /\FolStabilizerIsotId{\func}.
\end{equation}

Notice that if $\func\in\mathcal{F}(\Mman,\Pman)$ has only non-degenerate local extremes, e.g. it is Morse, then $\FolStabilizer{\func} = \ker(\rho)$.
Hence 
\begin{align*}
\StabilizerIsotId{\func} \cap \ker(\rho) &=
\Stabilizer{\func}\cap  \DiffIdM \cap \ker(\rho) =
%\DiffIdM \cap \ker(\rho) = \\ &=
 \DiffIdM \cap\FolStabilizer{\func} =
\FolStabilizerIsotId{\func},
\end{align*}
and therefore
\begin{align*}
\rho(\StabilizerIsotId{\func}) &\cong 
\frac{\StabilizerIsotId{\func}}{\StabilizerIsotId{\func} \cap \ker(\rho)} =
\frac{\StabilizerIsotId{\func}}{\FolStabilizerIsotId{\func}}
\cong \Gf,
\end{align*}
so in this case the definitions~\eqref{equ:Gf} and~\eqref{equ:Gf_1} agree, and $\Gf$ can be regarded as a subgroup of the group of automorphisms of the graph $\KRGraphf$ of $\func$.

However, for maps $\func\in\mathcal{F}(\Mman,\Pman)$ having degenerate local extremes $\FolStabilizer{\func}$ is a proper subgroup of $\ker(\rho)$.
Nevertheless, Lemmas~\ref{lm:Gf_prod_decomp} and~\ref{lm:Gf_on_disk_cyl} are proved in~\cite{Maksymenko:DefFuncI:2014} for all $\func\in\mathcal{F}(\Mman,\Pman)$ if one define $\Gf$ by~\eqref{equ:Gf_1}.

\begin{mytheorem}
Let $\Mman$ be a compact surface distinct from $S^2$ and $T^2$ and 
\[\mathbf{G}_{\mathcal{F}}(\Mman,\Pman) = \{ \StabilizerIsotId{\func} /\FolStabilizerIsotId{\func} \mid \func\in\mathcal{F}(\Mman,\Pman)\}\]
be the family of isomorphism classes of group $\Gf$ defined by~\eqref{equ:Gf_1}, where $\func$ run over all $\mathcal{F}(\Mman,\Pman)$.
Then $\mathbf{G}_{\mathcal{F}}(\Mman,\Pman) = \PClassAll$.
\end{mytheorem}
{\em Sketch of proof.}
By Theorem~\ref{th:ClassesGf} $\PClassAll = \ClassGfM  \subset \mathbf{G}_{\mathcal{F}}(\Mman,\Pman)$.
The inverse inclusion can be proved similarly to $\ClassGfM \subset \PClassAll$ using 
Lemmas~\ref{lm:Gf_prod_decomp} and~\ref{lm:Gf_on_disk_cyl} for $\mathcal{F}(\Mman,\Pman)$, and~\cite[Theorem~5.5]{Maksymenko:DefFuncI:2014} claiming that if $\func\in\mathcal{F}(D^2,\Pman)$ has a unique critical point and that point is a \myemph{degenerate local extreme}, then $\Gf \cong \bZ_{2n}$ for some $n\in\bN$.
We leave the details for the reader.
\myqed

%\bibliographystyle{plain}
%\bibliography{biblio}

\def\cprime{$'$}

\end{document}